\documentclass{article}
\usepackage[utf8]{inputenc}
\usepackage{graphicx}
\usepackage{amsfonts}
\usepackage{amsmath,bbm,bm}
\usepackage{amssymb}
\usepackage{fancyhdr}
\usepackage{titlesec}
\usepackage{indentfirst}
\usepackage{booktabs}
\usepackage{verbatim}
\usepackage{color}
\usepackage{amsthm}
\usepackage{url}
\usepackage{cite}
\usepackage{subfigure}
\usepackage[page,header]{appendix}
\usepackage{titletoc}
 \usepackage{tikz}
\usepackage{multirow}

\theoremstyle{plain}
\newtheorem{theorem}{Theorem}[section]

\newtheorem{corollary}[theorem]{Corollary}

\newtheorem{definition}{Definition}[section]

\newtheorem{remark}[theorem]{Remark}

\newcommand{\rd}{\,\mathrm{d}}

\newcommand{\mH}{\mathcal{H}}
\newcommand{\mV}{\mathcal{V}}
\newcommand{\mW}{\mathcal{W}}
\newcommand{\mM}{\mathcal{M}}
\newcommand{\mE}{\mathcal{E}}
\newcommand{\bx}{\mathbf{x}}
\newcommand{\bn}{\mathbf{n}}
\usepackage{ulem}

\begin{document}

	\title{Positivity-preserving and energy-dissipative \\   finite difference schemes \\for the Fokker-Planck and Keller-Segel equations}  	

	\author{Jingwei Hu\footnote{Department of Applied Mathematics, University of Washington, Seattle, WA 98195, USA (hujw@uw.edu). JH's research was supported by NSF CAREER grant DMS-2153208 and AFOSR grant FA9550-21-1-0358.} \  \  
	and \ Xiangxiong Zhang\footnote{Department of Mathematics, Purdue University, West Lafayette, IN 47907, USA (zhan1966@purdue.edu). XZ's research was supported by NSF grant DMS-1913120.} }    
	\maketitle
	
\begin{abstract}
In this work, we introduce semi-implicit or implicit finite difference schemes for the continuity equation with a gradient flow structure. Examples of such equations include the linear Fokker-Planck equation and the Keller-Segel equations. The two proposed schemes are first order accurate in time, explicitly solvable, and second order and fourth order accurate in space which are obtained via finite difference implementation of the classical continuous finite element method.  The fully discrete schemes are proved to be positivity-preserving and energy-dissipative: the second order scheme can achieve so unconditionally while the fourth order scheme only requires a mild time step and mesh size constraint. In particular, the fourth order scheme is the {\it first} high order spatial discretization that can achieve both positivity and energy decay properties, which is suitable for long time simulation and to obtain accurate steady state solutions.
\end{abstract}

{\small 

}


\section{Introduction}
\label{sec:intro}

In this paper, we are interested in the continuity equation of the form
\begin{align} 
&\partial_t \rho=\nabla \cdot [\rho\nabla (\mH'(\rho)+\mV+\mW*\rho)], \quad t> 0, \  \bx\in \Omega \subset \mathbb{R}^d, \label{conti} \\
&\rho(0,\bx)=\rho_0(\bx), 
\end{align}
where $\rho=\rho(t,\bx)\geq 0$ is the unknown density function, $\mH(\rho)$ is the internal energy which is assumed to be convex, $\mV(\bx)$ is the external potential, and $\mW(\bx)$ is the interaction potential. The typical boundary condition of (\ref{conti}) is the no-flux boundary:
\begin{equation} \label{boundary}
\nabla (\mH'(\rho)+\mV+\mW*\rho)\cdot \bn=0, \quad \bx\in\partial \Omega,
\end{equation}
where $\bn$ is the outward normal. Therefore, the total mass is conserved
\begin{equation*}
\int_\Omega\rho(t,\bx)\rd{\bx}=\int_\Omega\rho_0(\bx)\rd{\bx}.
\end{equation*}

Equations of the form (\ref{conti}) appear in various contexts, for example, in modeling of porous medium \cite{Vazquez07}, granular materials \cite{CMV03}, and collective behavior of biological and social systems \cite{CCY19}. In particular, we focus on the following two cases in this paper: the linear Fokker-Planck equation and the Keller-Segel model of chemotaxis. For both cases, the internal energy function is given by
\begin{equation} \label{HH}
\mH(\rho)=\rho \log \rho-\rho.
\end{equation} 

In the Fokker-Planck equation,
\begin{equation*}
\mV=\mV(\bx), \quad \mW\equiv 0,
\end{equation*}
where $\mV(\bx)$ is some given function bounded from below in $\Omega$. In this case, \eqref{conti} can also be written as  a convection-diffusion equation,
 \begin{equation} \label{FP}
 \partial_t\rho=\Delta \rho+\nabla \cdot( \rho \nabla \mV).
 \end{equation}

In the Keller-Segel model, $\rho$ is the density of some bacteria and 
\begin{equation*}
\mV\equiv0, \quad \mW*\rho=-c,
\end{equation*}
where $c=c(t,\bx)$ is the density of chemical attractant satisfying an elliptic equation in $\Omega$ with a constant $\alpha \geq 0$:
\begin{equation} \label{PP}
-\Delta c+\alpha c=\rho.
\end{equation} In this case, \eqref{conti} can be written as
 \begin{equation} \label{KS}
 \partial_t\rho=\Delta \rho-\nabla \cdot( \rho \nabla c),
 \end{equation}
 which is coupled with (\ref{PP}) to form a system.
Note that if $\Omega$ is $\mathbb{R}^d$, $\mW$ is the Newtonian potential when $\alpha=0$ and the Bessel potential when $\alpha>0$. 
By integrating (\ref{PP}) in $\Omega$, we obtain
\begin{equation*}
-\nabla c\cdot \bn\big|_{\partial \Omega}+\alpha \int_{\Omega} c\rd{\bx}= \int_{\Omega} \rho\rd{\bx}.
\end{equation*}
Therefore, the boundary condition of $c$ must be compatible with the equation above.
When $\alpha=0$,  the Neumann boundary condition must satisfy the compatibility condition
\begin{equation*}
-\nabla c\cdot \bn\big|_{\partial \Omega}= \int_{\Omega} \rho_0\rd{\bx}.
\end{equation*}
When $\alpha>0$, if we consider the homogeneous Neumann boundary $\nabla c\cdot \bn \big|_{\partial \Omega}=0$, then 
\begin{equation*}
\alpha \int_{\Omega} c\rd{\bx}=\int_{\Omega} \rho_0\rd{\bx},
\end{equation*}
i.e., the mass of $c$ is also conserved.  

The equation (\ref{conti}) has a variational structure. It is the gradient flow, with respect to the 2-Wasserstein metric, of the free energy functional \cite{Villani03}:
\begin{equation} \label{energy}
\mE(\rho)=\int_{\Omega}\left(\mH(\rho)+\mV\rho+\frac{1}{2}(\mW*\rho)\rho\right)\rd{\bx}.
\end{equation}
Indeed 
\begin{equation*}
\frac{\delta \mE}{\delta \rho}=\xi, \quad \xi:=\mH'(\rho)+\mV+\mW*\rho,
\end{equation*}
hence
\begin{equation}  \label{energy2}
\frac{\rd \mE}{\rd t}=\int_{\Omega}\frac{\delta \mE}{\delta \rho}\partial_t \rho\rd{\bx}=\int_{\Omega}\xi \nabla\cdot \left( \rho \nabla  \xi\right)\rd{\bx}=-\int_{\Omega} \rho |\nabla  \xi|^2\rd{\bx}\leq 0.
\end{equation}

Note that for $\mH$ given in (\ref{HH}), we can define
\begin{equation*}
\mathcal{M}=e^{\log \rho-\xi}=e^{-(\mV+\mW*\rho)}.
\end{equation*}
With this $\mM$, the equation (\ref{conti}) can be written equivalently as
\begin{equation} \label{eqn1}
\partial_t\rho=\nabla \cdot \left( \mM\nabla \left( \frac{\rho}{\mM}\right)\right).
\end{equation}
The boundary condition (\ref{boundary}) becomes
\begin{equation}
\nabla \left(\frac{\rho}{\mM} \right)\cdot \bn=0, \quad \bx\in\partial \Omega.
\label{eq-bc}
\end{equation}
Furthermore, the energy (\ref{energy}) can be written equivalently as
\begin{equation}  \label{energy1}
\mE(\rho)=\int_{\Omega} \left(\rho\log \left( \frac{\rho}{\mathcal{M}}\right)-\rho-\frac{1}{2}(\mW*\rho)\rho\right)\rd{\bx}.
\end{equation}
When written in form (\ref{eqn1}), the original continuity equation (\ref{conti}) can be viewed as a ``variable coefficient'' diffusion equation,  for which we are able to construct efficient positivity-preserving and energy-dissipative schemes, i.e., the discrete analog of (\ref{energy1}) is decreasing in time. 
In the literature there  are many numerical schemes for the Fokker-Planck or Keller-Segel type equations. Recently, significant efforts have been devoted to structure-preserving discretizations to preserve, for instance, the positivity of the solution and energy decay at the semi-discrete or fully discrete level. We summarize some of the recent methods according to their types of time discretization.  The first kind of methods are fully explicit schemes. For a scalar convection-diffusion equation such as \eqref{FP}, there are quite a few explicit positivity-preserving schemes \cite{zhang2013maximum, srinivasan2018positivity, li2018high, qiu2021third}, however with a small time step constraint $\Delta t=\mathcal O(\Delta x^2)$ which  is unacceptable in applications requiring long time simulation. Most importantly, it is usually quite difficult to establish energy dissipation in these positivity-preserving schemes. Some recent explicit schemes, including a finite volume method in \cite{CCH15} and discontinuous Galerkin methods in \cite{sun2018discontinuous, GLY19}, can indeed achieve energy dissipation but only in the semi-discrete setting (i.e., the time is left as continuous). The second kind of methods are implicit or semi-implicit nonlinear schemes. For such schemes, it is possible to preserve positivity and energy dissipation in the fully discrete setting without a small time step constraint \cite{Bubba, BCH20, SX20}, but they often involve nonlinear systems, for which robust nonlinear system solvers are needed.  The third kind of methods are implicit or semi-implicit schemes that are explicitly solvable.  By formulating the continuity equation as in (\ref{eqn1}) and treating $\mM$ explicitly, one can derive a semi-implicit scheme, in which only a linear system needs to be solved without small time-step constraint. Note that this approach is only possible for linear diffusions (for $\mH$ given by (\ref{HH})) and has been used in many previous works, for example, \cite{JY11, LWZ18, HS19, HH20, HLXZ21}. Although details vary, they all use the second order central finite difference for spatial discretization. We use the third approach for the time discretization in this paper. However, the proposed spatial discretization can achieve fourth order accuracy,  which is one of the main novelties.  Furthermore, we can prove the fully discrete positivity and energy decay property for the fourth order spatial discretization under reasonable mesh size and time step constraints. We emphasize that the time step constraint in this paper is a lower bound thus no small time-step constraint like $\Delta t=\mathcal O(\Delta x^2)$ is required. To the best of our knowledge, this is the first  high order spatial discretization that can achieve these properties  for the linear Fokker-Planck and Keller-Segel type equations.  


The rest of this paper organized as follows. In Section \ref{sec-scheme}, we introduce the finite difference schemes, which are obtained by finite difference implementation of continuous finite element method with the linear and quadratic polynomials. In Section \ref{sec-mono}, we show that both the second order and fourth order schemes are monotone. It is well known that the second order central difference or linear finite element method for linear diffusion forms an M-matrix thus is monotone. The fourth order accurate scheme or the finite element method with quadratic polynomial basis no longer gives an M-matrix but monotonicity can still be proved under practical mesh size and time step constraints. In Section \ref{sec-positivity}, we show that monotonicity implies positivity and fully discrete energy dissipation in these schemes. Section \ref{sec-test} includes numerical tests on the Fokker-Planck equation and Keller-Segel system. Concluding remarks are given in Section \ref{sec-remark}. 

\section{Finite difference schemes}
\label{sec-scheme}
In this section, we introduce a simple numerical scheme for equation (\ref{eqn1}) with a first order accurate semi-implicit time discretization. For the spatial discretization, we use second order and fourth order accurate finite difference schemes, which are obtained from finite element method using linear and quadratic polynomial bases respectively. It is well known that a finite element method with suitable quadrature is also a finite difference scheme. In particular, the fourth order accurate finite difference scheme considered here is equivalent to the Lagrangian $Q^2$ (tensor product of polynomials of degree $2$) finite element method with $3$-point Gauss-Lobatto quadrature, which is also known as the $Q^2$ spectral element method \cite{maday1990optimal}. The main novelty here is that we can  prove rigorous positivity-preserving and energy-dissipation properties for the fully discrete scheme,  especially the fourth order spatial discretization in one and two spatial dimensions. 

In this section, we mainly focus on how the finite difference schemes are defined. The explicit form of the schemes will be given in Section \ref{sec-mono}.
 We only consider one and two spatial dimensions in this paper, even though one can also derive these schemes in higher dimensions. 


\subsection{Time discretization}
\label{sec-timed-sovler}

We propose the following semi-implicit discretization of (\ref{eqn1}):
\begin{equation}
\frac{\rho^{n+1}-\rho^n}{\Delta t}=\nabla \cdot \left( \mathcal{M}^n\nabla \left( \frac{\rho^{n+1}}{\mathcal{M}^n}\right)\right), \quad \bx\in \Omega,
\label{eq-timescheme}
\end{equation}
where
\begin{equation*}
\mathcal{M}^n=e^{-(\mV+\mW*\rho^{n})}.
\end{equation*}
The no-flux boundary condition \eqref{eq-bc} is imposed as
\begin{equation} \label{eq-bc1}
\nabla \left(\frac{\rho^{n+1}}{\mM^n} \right)\cdot \bn=0, \quad \bx\in\partial \Omega.
\end{equation}
Note that (\ref{eq-timescheme}) is equivalent to 
\begin{equation*}
\frac{\rho^{n+1}-\rho^n}{\Delta t}=\nabla\cdot (\rho^{n+1}\nabla(\log \rho^{n+1}+\mV+\mW*\rho^n))
\end{equation*}
for discretizing the original equation \eqref{conti}.

We then introduce the auxiliary variables defined as
\begin{equation} \label{gg}
\tilde{g}^{n+1}:=\frac{\rho^{n+1}}{\mathcal M^n}, \quad g^{n}:=\frac{\rho^{n}}{\mathcal M^n}, 
\end{equation}
and write the scheme \eqref{eq-timescheme} as
\begin{equation}
\mathcal M^n\tilde{g}^{n+1}-\Delta t\nabla \cdot \left( \mathcal{M}^n\nabla  \tilde{g}^{n+1} \right)=\mathcal M^n g^n.
\label{eq-timescheme-g}
\end{equation}
Accordingly the boundary condition \eqref{eq-bc1} becomes the homogeneous Neumann boundary for the auxiliary variable: 
$$ \nabla \tilde{g}^{n+1}\cdot \bn=0,\quad\bx\in\partial \Omega.$$
After multiplying a test function $v\in H^1(\Omega)$ to \eqref{eq-timescheme-g} and integration by parts using the boundary condition for $\tilde{g}^{n+1}$, we obtain the variational form of \eqref{eq-timescheme-g}: seek $\tilde{g}^{n+1}\in H^1(\Omega)$ that satisfies
\begin{equation*}
(\mathcal M^n \tilde{g}^{n+1}, v)+\Delta t(\mathcal M^n \nabla  \tilde{g}^{n+1},\nabla  v)=(\mathcal M^n g^{n}, v),\quad\forall v\in H^1(\Omega),
\end{equation*}
where $(v,w):=\int_{\Omega} v w \rd{\bx} $ denotes the $L^2$ inner product in $\Omega$.

\begin{remark}
In the Fokker-Planck equation, $\mM=\exp(-\mV(\bx))$ is a time-independent quantity and (\ref{eq-timescheme}) simplifies to a fully implicit scheme. For brevity, our following presentation will focus on the Keller-Segel equation for which $\mM^n=\exp(c^n(\bx))$. Reduction to the Fokker-Planck case will be commented whenever necessary.
\end{remark}

\subsection{Spatial discretization}
\label{sec-derivation}

We consider a uniform rectangular mesh $\Omega_h$ for the rectangular domain $\Omega$. For any rectangle $e$ in the mesh $\Omega_h$, let $Q^k$ be the space of tensor product polynomials of degree $k$. For instance, in two dimensions, 
$$Q^k(e)=\left\{ p(x,y)=\sum\limits_{i=0}^k\sum\limits_{j=0}^k p_{ij}x^i y^j, (x,y)\in e \right\}.$$ 
Let $V^h$ be the continuous piecewise $Q^k$ polynomial space defined on $\Omega_h$:
\[V^h=\{ v_h(\mathbf x) \in C(\Omega): v_h \big|_e\in Q^k(e),\forall e \in \Omega_h\}.\]
 The $Q^k$ finite element method for \eqref{eq-timescheme-g} is to seek $\tilde{g}^{n+1}_h\in V^h$ satisfying
\begin{equation}
(\mathcal M^n \tilde{g}^{n+1}_h, v_h)+\Delta t(\mathcal M^n \nabla  \tilde{g}^{n+1}_h,\nabla  v_h)=(\mathcal M^n g^{n}_h, v_h),\quad\forall v_h\in V^h,
 \label{zxxeqn-fem}
\end{equation} 
where $\mathcal M^n$ is regarded as a given variable coefficient at time step $n$.
 
The $Q^k$ spectral element method is to replace all integrals in \eqref{zxxeqn-fem} by $m$-point Gauss-Lobatto quadrature with $m\geq k+1$ in each dimension. Standard finite element method error estimates still hold if $m\geq k+1$, i.e., the $Q^k$ spectral element method is $(k+1)$-th order accurate in $L^2$-norm and $k$-th order accurate in $H^1$-norm for smooth solutions of an elliptic equation, see \cite{maday1990optimal}. We consider the simplest choice of quadrature, using $(k+1)$-point Gauss-Lobatto quadrature. Then the method is to find $\tilde{g}^{n+1}_h\in V^h$ satisfying
\begin{equation} \langle \mathcal M^n \tilde{g}^{n+1}_h, v_h\rangle+ \Delta t \langle \mathcal M^n \nabla  \tilde{g}^{n+1}_h , \nabla  v_h\rangle=\langle \mathcal M^n g^{n}_h, v_h\rangle ,\quad\forall v_h\in V^h,
\label{zxxeqn-fem2}
\end{equation}
where $\langle\cdot,\cdot\rangle$ denotes that integrals are replaced by $(k+1)$-point Gauss-Lobatto quadrature. 

For a two-dimensional problem, a $Q^k$ polynomial on a rectangular element $e$ can be represented as a Lagrangian interpolation polynomial at $(k+1)\times (k+1)$ Gauss-Lobatto points, thus all Gauss-Lobatto points in \eqref{zxxeqn-fem2} are not only quadrature nodes but also nodes representing all degrees of freedom. So the $Q^k$ spectral element method \eqref{zxxeqn-fem2} also becomes a finite difference scheme on all Gauss-Lobatto nodes. For $k\geq 3$, the Gauss-Lobatto points are not uniform in each element. For $k\leq 2$, all Gauss-Lobatto nodes on $\Omega_h$ correspond to a uniform grid, see Figure \ref{mesh} for an illustration of the $Q^2$ mesh. Moreover, for $k\geq 2$, such a finite difference scheme can be proved to be $(k+2)$-order accurate in discrete $l^2$-norm for elliptic equations \cite{li2020superconvergence} and for parabolic equations \cite{li2021accuracy}, e.g., the $Q^2$ spectral element method can be regarded as a fourth order accurate finite difference scheme.

 \begin{figure}[h]
 \subfigure[All  quadrature points on $\Omega_h$]{\includegraphics[scale=0.8]{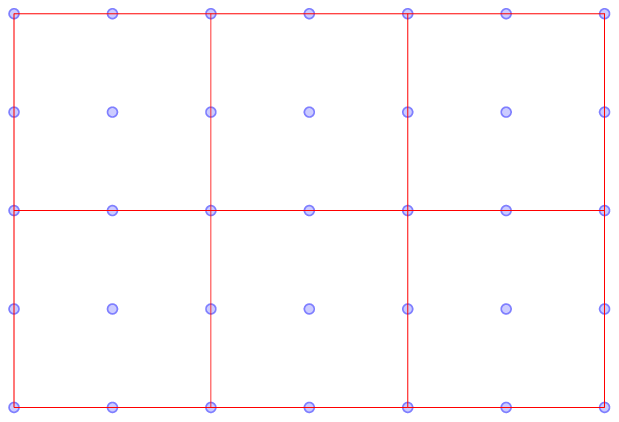} }
 \hspace{.6in}
\subfigure[The corresponding finite difference grid on $\Omega_h$]{\includegraphics[scale=0.8]{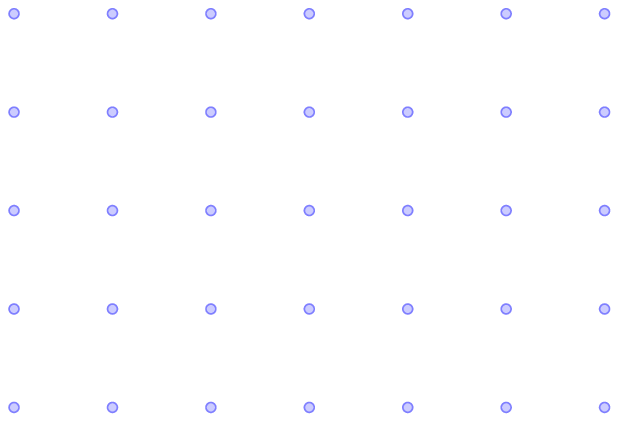}}
\caption{The $3\times 3$ Gauss-Lobatto quadrature points for $Q^2$ finite element method on a uniform mesh  $\Omega_h$ naturally gives a uniform finite difference grid.  }
\label{mesh}
 \end{figure}

In this paper, we only consider the linear case $k=1$ and quadratic case $k=2$, because only in these two cases the schemes can be proved to be positivity-preserving and energy-dissipative. To derive an equivalent matrix form of the scheme \eqref{zxxeqn-fem2}, let $\phi_i(\bx)$ ($i=1,\cdots,N$) be the $Q^k$ Lagrangian basis at all Gauss-Lobatto points $\bx_i$ ($i=1,\cdots,N$) on $\Omega_h$. For any piecewise polynomial $u_h(\bx)\in V^h$, let $u_i=u_h(\bx_i)$. Then $u_h(\bx)=\sum\limits_{i=1}^N u_i\phi_i(\bx)$. Let $\mathbf u=\begin{bmatrix} u_1\\ \vdots \\u_N \end{bmatrix}$ and $w_i$ be the quadrature weight at $\bx_i$.  
 
With the notation above, we have
\begin{equation} \label{part1}
\langle \mathcal M^n \tilde{g}^{n+1}_h, v_h\rangle =\sum_{i=1}^N w_i \mathcal M^n_i \tilde{g}^{n+1}_i v_i =\mathbf v^T W   M^n \mathbf{g}^{n+1},
\end{equation}
where $W=\text{diag}\{w_1, \cdots ,w_N\}$ and $M^n=\text{diag}\{\mathcal M^n_1,\cdots, \mathcal M^n_N\}$ are diagonal matrices. We also have
\begin{equation} \label{part2}
\langle \mathcal M^n \nabla  \tilde{g}^{n+1}_h,\nabla  v_h\rangle= \mathbf v^T S \tilde{\mathbf g}^{n+1},
\end{equation}
where $S$ is the stiffness matrix from the same spectral element method solving a Poisson equation $-\nabla\cdot (\mathcal M^n\nabla u)=f$ in $\Omega$ with homogeneous Neumann boundary condition $\nabla u\cdot \mathbf{n}=0$ on $\partial \Omega$. In other words, $S$ is the  stiffness matrix  in the scheme of seeking $u_h\in V^h$ satisfying
\[ \langle \mathcal M^n \nabla u_h, \nabla v_h\rangle =\langle f, v_h\rangle,\quad\forall v_h\in V^h. \]
We emphasize that the stiffness matrix $S$ depends on $\mathcal M^n_i>0$. It is common knowledge in finite element theory that $S$ satisfies two properties:
\begin{enumerate}
\item $S$ is real symmetric and positive semi-definite.
\item Its null space is one-dimensional and the null vector is $\mathbf 1$.
\end{enumerate}
Here for brevity, we do not give the explicit form of $S$. The complete scheme \eqref{zxxeqn-fem2}  in one and two dimensions will be given in Section \ref{sec-mono}. 

Using (\ref{part1}) and (\ref{part2}), the finite difference scheme (\ref{zxxeqn-fem2}) can be written in the matrix form as: find $\tilde{\mathbf g}^{n+1}$ satisfying
\begin{equation}
\mathbf v^T W M^n \tilde{\mathbf g}^{n+1}+\Delta t   \mathbf v^T S \tilde{\mathbf g}^{n+1}=\mathbf v^T WM^n \mathbf g^{n} ,\quad\forall \mathbf v\in \mathbbm R^N,
 \label{zxxeqn-fd}
\end{equation}
or equivalently
\begin{equation}
W M^n \tilde{\mathbf g}^{n+1}+\Delta t     S \tilde{\mathbf g}^{n+1}=  WM^n\mathbf g^{n} .
 \label{zxxeqn-fd3}
\end{equation} 
Noticing (\ref{gg}), \eqref{zxxeqn-fd3} can also be written as
\begin{equation}
 W\boldsymbol{\rho}^{n+1}+\Delta t  S (M^n)^{-1} \boldsymbol{\rho}^{n+1}=  W\boldsymbol{\rho}^{n} .
 \label{zxxeqn-fd2}
\end{equation}

\begin{remark}
Even though the scheme \eqref{zxxeqn-fd2} for $\boldsymbol{\rho}$ does not involve any auxiliary variable $\mathbf g$, the division by $\mathcal M^n_i$ is still needed
in \eqref{zxxeqn-fd2}. Moreover,  \eqref{zxxeqn-fd3} gives a symmetric positive definite linear system but  \eqref{zxxeqn-fd2} does not. In practice, both can be solved by preconditioned conjugate gradient methods with efficient inversion of Laplacian as a preconditioner, see Section 7 in  \cite{li2020superconvergence} for implementation details. In our numerical tests, we solve the system \eqref{zxxeqn-fd3}  by preconditioned conjugate gradient.
 \end{remark}

\subsection{The full scheme for the Keller-Segel system}

In the case of the Keller-Segel system, in addition to (\ref{zxxeqn-fd3}) (the discretization for (\ref{KS})) one also needs to discretize the equation (\ref{PP}). Here we consider $\alpha>0$ and the homogeneous Neumann boundary condition $\nabla c\cdot {\bf n}|_{\partial \Omega}=0$. We use the same scheme as in (\ref{zxxeqn-fem2}): find $c^n_h\in V^h$ satisfying
\begin{equation}
\label{scheme-c-n}
 \langle \nabla c_h^n, \nabla v_h \rangle +\alpha\langle  c_h^n,  v_h \rangle= \langle  \rho^n,  v_h \rangle,\quad \forall v_h\in V^h.\end{equation}
Similarly as in the previous subsection, \eqref{scheme-c-n} can be written equivalently in the finite difference or matrix form. 

In one dimension, the second order scheme ($k=1$) can be written as 
\[ \frac{1}{h^2}K\mathbf c^n+\alpha \mathbf c^n=\boldsymbol{\rho}^n,\]
and the fourth order scheme ($k=2$) can be written as
\[\frac{1}{h^2} H\mathbf c^n+\alpha \mathbf c^n=\boldsymbol{\rho}^n,\]
where $h$ is the grid spacing and 
\[K=\left(\begin{smallmatrix}
              2 & -2 & & & &\\
              -1 & 2 & -1 & & & \\              
              & -1 & 2 & -1 & & \\
              & & \ddots & \ddots & \ddots & \\
              & & & -1 & 2 &-1 \\
              & & & & -2 & 2 \\
             \end{smallmatrix}\right)_{N\times N},\quad H=\left(\begin{smallmatrix}
     \frac72 &-4 & \frac12 & & &\\
       -1 & 2& -1 & & &\\
    \frac14 &-2& \frac72 &-2 & \frac14 & \\
   &   &  -1 & 2& -1 & \\
   & & &\ddots &\ddots & \ddots\\
&&   &    \frac14 &-2& \frac72 &-2 & \frac14\\
  &    & & & &-1 & 2 &-1\\
&  && & &\frac12 &-4& \frac72 \\
  \end{smallmatrix}\right)_{N\times N}.\]
We emphasize that  $N$ must be odd in the matrix $H$ for the fourth order scheme because the grid points are from Gauss-Lobatto nodes, see Figure \ref{mesh}. 

In two dimensions, let $\mathbf c$ be a two-dimensional array with $\mathbf c_{ij}$ denoting $(i,j)$ point value. Let $vec(\mathbf c)$ be a column vector obtained by rearranging entries in $\mathbf c$ column by column. Then the second order and fourth order schemes can be written, respectively, as 
\[\frac{1}{h^2}(K\otimes K) vec(\mathbf c^n)+\alpha vec(\mathbf c^n)=  \boldsymbol{\rho}^n, \]
and
\[\frac{1}{h^2}(H\otimes H) vec(\mathbf c^n)+\alpha vec(\mathbf c^n)=  \boldsymbol{\rho}^n.\]


To summarize, the full finite difference scheme for the Keller-Segel system (\ref{PP})-(\ref{KS}) is implemented as follows:
\begin{enumerate}
\item At time level $n$, given point values $\rho^n_i$ at each node $\mathbf x_i$, solve \eqref{scheme-c-n} to obtain $c^n_i$, then compute point values of $\mathcal M^n_i=\exp (c^n_i)$. In multiple dimensions, the linear system can be easily and efficiently inverted by eigenvalue decomposition of $K$ and $H$, see  \cite{li2020superconvergence} for details. 
\item With point values $g^n_i:=\frac{\rho^n_i}{\mathcal M^n_i }$, obtain $\tilde g_i^{n+1}$ by solving \eqref{zxxeqn-fd3}. 
\item Update $\rho$ by $\rho^{n+1}_i:=\mathcal M_i^n\tilde  g_i^{n+1}$.
\end{enumerate}

\begin{remark}
The finite difference scheme for the Fokker-Planck equation (\ref{FP}) is simpler: at each node $\mathbf x_i$, first define $\mathcal M_i=\exp (-\mV_i)$.
\begin{enumerate}
\item At time level $n$, given point values $\rho^n_i$, compute $g^n_i:=\frac{\rho^n_i}{\mathcal M_i }$, then obtain $\tilde g_i^{n+1}$ by solving \eqref{zxxeqn-fd3}. 
\item Update $\rho$ by $\rho^{n+1}_i:=\mathcal M_i \tilde g_i^{n+1}$.
\end{enumerate}

\end{remark}


\subsection{Accuracy of the spatial discretization}

For the $Q^2$ finite element method with 3-point Gauss-Lobatto quadrature, it is well known that the standard $L^2$-norm error estimate is third order. However, when regarded as a finite difference scheme at Gauss-Lobatto points, it can be rigorously proved that it is a fourth order accurate scheme in the discrete $l^2$-norm  \cite{li2020superconvergence, li2021accuracy}. In particular, this has been proved for Dirichlet boundary conditions in  \cite{li2020superconvergence}. Only $\mathcal O(h^{3.5})$ can be proved for  Neumann boundary conditions for an operator like $-\nabla (A(\mathbf x)\nabla u)$ where $A(\mathbf x)$ is a positive definite matrix, and the one half order loss is purely due to the mixed second order derivatives. Nonetheless, for the equations we are interested in here, i.e., an operator  like $-\nabla \cdot (a(\mathbf x)\nabla u)$
with a scalar coefficient $a(\mathbf x)$, since there are no mixed second order derivatives involved, the same proof in \cite{li2020superconvergence, li2021accuracy} applies to show that the fourth order accuracy also holds for Neumann boundary conditions of elliptic equations, see \cite{phdthesis} for a detailed proof.f
 So for both \eqref{zxxeqn-fd3} and \eqref{scheme-c-n}, we will refer to the $Q^2$ scheme as the fourth order accurate spatial discretization, i.e., it is a fourth order accurate scheme for solving a steady state problem.

For the $Q^1$ finite element method  with quadrature, it is also well known that it gives the most popular second order central finite difference scheme. However, for the Neumann boundary condition, there is still some subtle difference, which will be reviewed in Remark \ref{remark-accuracy} of Section \ref{sec-mono}.


\section{Monotonicity of the finite difference schemes}
\label{sec-mono}

A matrix $A$ is called {\it monotone} if its inverse has nonnegative entries $A^{-1}\geq 0$. In this section we discuss the monotonicity of the matrix used in the second order and fourth order finite difference schemes \eqref{zxxeqn-fem2}, which is the key intrinsic property implying positivity and energy dissipation.  

In particular, we consider the matrix form \eqref{zxxeqn-fd3}, 
which can also be written as 
\begin{equation}
(M^n +\Delta t    W^{-1} S \mathbf) {\bf \tilde g}^{n+1}=  M^n\mathbf g^{n} .
 \label{zxxeqn-fd4}
\end{equation}
We will discuss  the monotonicity of the matrix $M^n+\Delta t  W^{-1}S$.  For simplicity, we will drop superscript $n$ in $M$ in the rest of this section. 

For the second order scheme, it is well known that it forms an M-matrix thus is monotone, which will be reviewed. For the fourth order scheme, the monotonicity for Dirichlet boundary condition in two dimensions was proved in \cite{li2019monotonicity}. The same results in  \cite{li2019monotonicity} also hold for the Neumann boundary conditions. For completeness, in this section we include a detailed proof for the monotonicity of the fourth order scheme \eqref{zxxeqn-fd4} with the homogeneous Neumann boundary condition  for $\tilde{g}^{n+1}$, which is equivalent to the no-flux boundary condition for $\rho^{n+1}$. 

\subsection{M-matrices}
 The only viable tool in the literature to prove monotonicity is to use M-matrices. 
 Nonsingular M-matrices are monotone matrices and there are many equivalent definitions or characterizations of M-matrices, see 
\cite{plemmons1977m}. 
 By condition $K_{35}$ in \cite{plemmons1977m}, a sufficient and necessary characterization  is,
 \begin{theorem}
\label{rowsumcondition-thm2}
 For a real square matrix $A$  with positive diagonal entries and non-positive off-diagonal entries, $A$ is a nonsingular M-matrix if  and only if there exists a positive diagonal matrix 
 $D$ such that $AD$ has all
positive row sums.
 \end{theorem}
The following is a convenient sufficient but not necessary characterization of nonsingular M-matrices \cite{li2019monotonicity}:
\begin{theorem}
\label{rowsumcondition-thm}
For a real square matrix $A$  with positive diagonal entries and non-positive off-diagonal entries, $A$ is a nonsingular M-matrix if  all the row sums of $A$ are non-negative and at least one row sum is positive. 
\end{theorem}

 \subsection{The second order  scheme in one dimension}
 \label{sec-1d-2ndscheme}
In the one dimensional case, assume the domain is $\Omega=[-L, L]$ and 
 the uniform grid points are $-L=x_1<x_2<\cdots<x_N=L$ with grid spacing $h$.  
  Following derivations in Section 7 of \cite{li2020superconvergence}, it is straightforward to show that the linear finite element method \eqref{zxxeqn-fd4} with a variable coefficient $\mathcal M>0$ 
  can be explicitly written as:
\begin{equation}
\resizebox{\hsize}{!}{$
\begin{split}
& \mathcal M_1 \tilde{g}^{n+1}_1+\Delta t  \frac{(\mathcal M_{1}+\mathcal M_2)\tilde{g}^{n+1}_1-(\mathcal M_{1}+\mathcal M_{2})\tilde{g}^{n+1}_{2}}{h^2   }=\mathcal M_1 g^{n}_1;\\
&  \mathcal M_i \tilde{g}^{n+1}_i + \Delta t  \frac{-(\mathcal M_{i-1}+\mathcal M_{i})\tilde{g}^{n+1}_{i-1}+(\mathcal M_{i-1}+2\mathcal M_i+\mathcal M_{i+1})\tilde{g}^{n+1}_i-(\mathcal M_{i}+\mathcal M_{i+1})\tilde{g}^{n+1}_{i+1}}{2h^2}\\  \\
&=  \mathcal M_i g^{n}_i, \quad  i=2,\cdots, N-1;\\
&\mathcal M_N  \tilde{g}^{n+1}_N+\Delta t  \frac{-(\mathcal M_{N-1}+\mathcal M_{N})\tilde{g}^{n+1}_{N-1}+(\mathcal M_{N-1}+\mathcal M_N)\tilde{g}^{n+1}_N}{h^2  }= \mathcal M_N g^{n}_N.
\end{split}
$}
\label{2nd-fd-fem}
\end{equation} 
It is easy to see that $M^n+\Delta t  W^{-1}S$ is a tridiagonal matrix satisfying Theorem \ref{rowsumcondition-thm}, thus is a nonsingular M-matrix and monotone.

Now for the ease of presentation of the scheme, we will abuse notation by introducing ghost point values as $\tilde g_0^{n+1}:=\tilde g_2^{n+1}$, $\tilde g_{N+1}^{n+1}:=\tilde g_{N-1}^{n+1}$ and $\mathcal M_0:=\mathcal M_2$,
$\mathcal M_{N+1}:=\mathcal M_{N-1}$. Then the scheme can be equivalently written as
\begin{equation}
\resizebox{\hsize}{!}{$
\begin{split}
&\mathcal M_i \tilde{g}^{n+1}_i+\Delta t  \frac{-(\mathcal M_{i-1}+\mathcal M_{i})\tilde{g}^{n+1}_{i-1}+(\mathcal M_{i-1}+2\mathcal M_i+\mathcal M_{i+1})\tilde{g}^{n+1}_i-(\mathcal M_{i}+\mathcal M_{i+1})\tilde{g}^{n+1}_{i+1}}{2h^2}\\
&=  \mathcal M_i g^{n}_i,   \quad i=1,\cdots, N.
\end{split}$}
\end{equation}
We emphasize that the scheme still has a different structure at the boundary points, and here ghost points are used only for a uniform expression of the scheme. 
In actual implementation, there are no ghost points.

\begin{remark}
\label{remark-accuracy}
One popular finite difference method to solve \eqref{eq-timescheme} is to apply the central finite difference as 
 \[\frac{\rho^{n+1}_i-\rho^n_i}{\Delta t}=\frac{F^{n+1}_{i+\frac12}-F^{n+1}_{i-\frac12}}{h},\]
 with the flux term defined by
 \[F^{n+1}_{i+\frac12}=\frac{1}{h} \frac{\mathcal M_i+\mathcal M_{i+1}}{2}
 \left(\frac{\rho^{n+1}_{i+1}}{\mathcal M_{i+1}}-\frac{\rho^{n+1}_{i}}{\mathcal M_{i}}\right),\]
 which is equivalent to 
 \[\tilde{g}^{n+1}_{i}-g^{n}_{i}=\frac{\Delta t}{h\mathcal M_i}(G^{n+1}_{i+\frac12}-G^{n+1}_{i-\frac12}),\quad  G^{n+1}_{i+\frac12}=\frac{1}{h} \frac{\mathcal M_i+\mathcal M_{i+1}}{2}
 \left(\tilde{g}^{n+1}_{i+1} -\tilde{g}^{n+1}_{i}\right).\]
For approximating no-flux boundary condition, if simply setting $G^{n+1}_{\frac12}=G^{n+1}_{N+\frac12}=0$, then the scheme becomes
\begin{equation}
\resizebox{\hsize}{!}{$
\begin{split}
&  \mathcal M_1 \tilde{g}^{n+1}_1+\Delta t  \frac{(\mathcal M_{1}+\mathcal M_2)\tilde{g}^{n+1}_1-(\mathcal M_{1}+\mathcal M_{2})\tilde{g}^{n+1}_{2}}{2h^2   }=\mathcal M_1 g^{n}_1;\\
& \mathcal M_i \tilde{g}^{n+1}_i + \Delta t  \frac{-(\mathcal M_{i-1}+\mathcal M_{i})\tilde{g}^{n+1}_{i-1}+(\mathcal M_{i-1}+2\mathcal M_i+\mathcal M_{i+1})\tilde{g}^{n+1}_i-(\mathcal M_{i}+\mathcal M_{i+1})\tilde{g}^{n+1}_{i+1}}{2h^2}\\
&=  \mathcal M_i g^{n}_i, \quad i=2,\cdots, N-1;\\
&\mathcal M_N  \tilde{g}^{n+1}_N + \Delta t  \frac{-(\mathcal M_{N-1}+\mathcal M_{N})\tilde{g}^{n+1}_{N-1}+(\mathcal M_{N-1}+\mathcal M_N)\tilde{g}^{n+1}_N}{2h^2  }= \mathcal M_N g^{n}_N.
\end{split}$}
\label{2nd-fd-center}
\end{equation} 
If using the same grid $-L=x_1<x_2<\cdots<x_N=L$ with grid spacing $h$, the scheme (\ref{2nd-fd-center}) is the same as \eqref{2nd-fd-fem} at interior points. For boundary points, \eqref{2nd-fd-center} is only first order accurate, which can be easily verified for constant coefficient case $\mathcal M_i\equiv 1$. If redefining $g_i$ and $\mathcal M_i$ as point values at a staggered uniform grid  $-L+\frac{h}{2}=x_1<x_2<\cdots<x_N=L-\frac{h}{2}$ with spacing $h$ (as has been done in most papers in the past, e.g. \cite{HH20}), the scheme  \eqref{2nd-fd-center} exhibits second order accuracy in many numerical tests. However, even on the staggered grid, the local truncation error of   \eqref{2nd-fd-center}  at $x_1=-L+\frac{h}{2}$ and $x_N=L-\frac{h}{2}$ is only first order,
thus it is quite difficult to rigorously prove the second order accuracy of \eqref{2nd-fd-center} by conventional finite difference analysis. 
On the other hand,  it can be easily proved that \eqref{2nd-fd-fem}  is second order accurate by standard finite element analysis. 
\end{remark}

 \subsection{The second order scheme in multiple dimensions}
 \label{sec-2D-2ndscheme}
In the two dimensional case, assume the domain is $\Omega=[-L, L]\times[-L,L]$ with 
an uniform $N\times N$ grid point with spacing $h$, which is a tensor product of the grid $-L=x_1<x_2<\cdots<x_N=L$.  
Let $\mathbf g$ be a $N\times N$ matrix with $g_{ij}$ denoting the point value at the $(i,j)$ grid point.  

We introduce the ghost values for $ i, j=1,\cdots, N$ as:
\begin{align*}
&\tilde g_{0,j}^{n+1}:=\tilde g_{2,j}^{n+1}, \ \tilde g_{N+1,j}^{n+1}:=\tilde g_{N-1,j}^{n+1},\  \tilde g_{i,0}^{n+1}:=\tilde g_{i,2}^{n+1}, \ \tilde g_{i,N+1}^{n+1}:=\tilde g_{i,N-1}^{n+1},\\
&\mathcal M_{0,j} :=\mathcal M_{2,j}, \ \mathcal M_{N+1, j} :=\mathcal M_{N-1,j},\  \mathcal M_{i,0} :=\mathcal M_{i,2}, \ \mathcal M_{i,N+1} :=\mathcal M_{i, N-1}.
\end{align*}

Then the  Lagrangian $Q^1$ finite element method  with 2-point Gauss Lobatto quadrature \eqref{zxxeqn-fem2}   can be explicitly expressed as 
\begin{align*}
\resizebox{\hsize}{!}{$\Delta t  \frac{-(\mathcal M_{i-1,j}+\mathcal M_{i j})\tilde{g}^{n+1}_{i-1,j}+(\mathcal M_{i-1,j}+2\mathcal M_{i j}+\mathcal M_{i+1,j})\tilde{g}^{n+1}_{ij}-(\mathcal M_{ij}+\mathcal M_{i+1,j})\tilde{g}^{n+1}_{i+1,j}}{2h^2}$}\\
\resizebox{\hsize}{!}{$+\Delta t  \frac{-(\mathcal M_{i,j-1}+\mathcal M_{i j})\tilde{g}^{n+1}_{i,j-1}+(\mathcal M_{i,j-1}+2\mathcal M_{i j}+\mathcal M_{i,j+1})\tilde{g}^{n+1}_{ij}-(\mathcal M_{ij}+\mathcal M_{i,j+1})\tilde{g}^{n+1}_{i,j+1}}{2h^2}$}\\
+  \mathcal M_{ij} \tilde{g}^{n+1}_{ij}=  \mathcal M_{ij} g^{n}_{ij}, \quad \forall  i, j=1,\cdots, N. 
 \end{align*}
 It is easy to see that $M^n+\Delta t  W^{-1}S$ is a matrix satisfying Theorem \ref{rowsumcondition-thm}, thus is a nonsingular M-matrix and monotone.
 
\begin{remark}
The scheme in three dimensional case can be similarly written and it is also straightforward to verify that $M^n+\Delta t  W^{-1}S$ is a matrix satisfying Theorem \ref{rowsumcondition-thm}, thus is a nonsingular M-matrix and monotone.
\end{remark}

\begin{remark}
We have seen that using the formulation (\ref{eqn1}) the second order finite difference scheme with a semi-implicit time discretization is unconditionally monotone thus always positivity-preserving and energy-dissipative (details to be given in Section~\ref{sec-positivity}). This is true even for blow-up solutions. As a comparison, for the Keller-Segel equation one can also use the formulation \eqref{KS} and apply the second order finite difference for both convection and diffusion operators with a semi-implicit time discretization, but the monotonicity can only be proved under a mesh constraint $h \|\nabla c\|_\infty \leq 2$. This is one of the key advantages of solving (\ref{eqn1}) instead of \eqref{KS}. 
\end{remark}

\subsection{Lorenz's condition for monotonicity}

For high order accurate schemes, especially for a variable coefficient problem, the stiffness matrices are no longer M-matrices. Yet, 
it is possible to show that the stiffness matrix is a product of two or more M-matrices thus still monotone \cite{li2019monotonicity, cross2020monotonicity} by using  the Lorenz's Theorem in \cite{lorenz1977inversmonotonie}, which will be briefly reviewed in this subsection. 

\begin{definition}
Let $\mathcal N = \{1,2,\dots,n\}$. For $\mathcal N_1, \mathcal N_2 \subset \mathcal N$, we say a matrix $A$ of size $n\times n$ connects $\mathcal N_1$ with $\mathcal N_2$ if 
\begin{equation}
\forall i_0 \in \mathcal N_1, \exists i_r\in \mathcal N_2, \exists i_1,\dots,i_{r-1}\in \mathcal N \quad \mbox{s.t.}\quad  a_{i_{k-1}i_k}\neq 0,\quad k=1,\cdots,r.
\label{condition-connect}
\end{equation}
If perceiving $A$ as a directed graph adjacency matrix of vertices labeled by $\mathcal N$, then \eqref{condition-connect} simply means that there exists a directed path from any vertex in $\mathcal N_1$ to at least one vertex in $\mathcal N_2$.  
In particular, if $\mathcal N_1=\emptyset$, then any matrix $A$  connects $\mathcal N_1$ with $\mathcal N_2$.
\end{definition}

Given a square matrix $A$ and a column vector $\mathbf x$, we define
\[\mathcal N^0(A\mathbf x)=\{i: (A\mathbf x)_i=0\},\quad 
\mathcal N^+(A\mathbf x)=\{i: (A\mathbf x)_i>0\}.\]

Given a matrix $A=[a_{ij}]\in \mathbbm{R}^{n\times n}$, define its diagonal, off-diagonal, positive and negative off-diagonal parts as $n\times n$ matrices $A_d$, $A_a$, $A_a^+$, $A_a^-$:
\[(A_d)_{ij}=\begin{cases}
a_{ii}, & \mbox{if} \quad i=j\\
0, & \mbox{if} \quad  i\neq j
\end{cases}, \quad A_a=A-A_d,
\]
\[(A_a^+)_{ij}=\begin{cases}
a_{ij}, & \mbox{if} \quad a_{ij}>0,\quad i\neq j\\
0, & \mbox{otherwise}.
\end{cases}, \quad A_a^-=A_a-A^+_a.
\]

The following two results were proved in  \cite{lorenz1977inversmonotonie}. See also \cite{li2019monotonicity} for a detailed proof.
\begin{theorem}\label{thm2}
If $A\leq M_1M_2\cdots M_k L$ where $M_1, \cdots, M_k$ are nonsingular M-matrices and $L_a\leq 0$,  and there exists a nonzero vector $\mathbf e\geq 0$ such that one of the matrices $M_1, \cdots, M_k , L$ connects $\mathcal N^0(A\mathbf e)$ with $\mathcal N^+(A\mathbf e)$. Then $M_k^{-1}M_{k-1}^{-1}\cdots M_1^{-1} A$ is an M-matrix, thus $A$ is a product of $k+1$ nonsingular M-matrices and $A^{-1}\geq 0$. 
\end{theorem}
\begin{theorem}[Lorenz's condition] \label{thm3}
If $A^-_a$ has a decomposition: $A^-_a = A^z + A^s = (a_{ij}^z) + (a_{ij}^s)$ with $A^s\leq 0$ and $A^z \leq 0$, such that 
\begin{subequations}
 \label{lorenz-condition}
\begin{align}
& A_d + A^z \textrm{ is a nonsingular M-matrix},\label{cond1}\\ 
& A^+_a \leq A^zA^{-1}_dA^s \textrm{ or equivalently } \forall a_{ij} > 0 \textrm{ with } i \neq j, a_{ij} \leq \sum_{k=1}^n a_{ik}^za_{kk}^{-1}a_{kj}^s,\label{cond2}\\
& \exists \mathbf e \in \mathbbm{R}^n\setminus\{\mathbf 0\}, \mathbf e\geq 0 \textrm{ with $A\mathbf e \geq 0$ s.t. $A^z$ or $A^s$  connects $\mathcal N^0(A\mathbf e)$ with $\mathcal N^+(A\mathbf e)$.} \label{cond3}
\end{align}
\end{subequations}
Then $A$ is a product of two nonsingular M-matrices thus $A^{-1}\geq 0$.
\end{theorem}

It was proved in \cite{cross2020monotonicity} that 
\begin{corollary} The matrix $L$ in Theorem \ref{thm2} must be an M-matrix.
\end{corollary}

In practice, the condition \eqref{cond3} can be difficult to verify. For the scheme we are interested in here, the vector $\mathbf e$ can be taken as $\mathbf 1$ consisting of all ones, then the condition \eqref{cond3} can be simplified. For the scheme \eqref{zxxeqn-fd4}, as long as $\mathcal M_i>0$, we always have $A\mathbf 1>0$ thus $N^0(A\mathbf 1)=\emptyset$ and \eqref{cond3} is trivially satisfied. We summarize it as follows:
\begin{theorem}
 \label{newthm3}
Let $A$ denote the matrix representation of the fourth order finite difference scheme obtained from Lagrangain $Q^2$ finite element method with $3$-point Gauss-Lobatto qudarture  solving $-\nabla\cdot( b\nabla)u+cu=f$ with variable coefficients $b>0$ and $c>0$ and homogeneous Neumann boundary condition in a rectangular domain. 
Assume $A^-_a$ has a decomposition $A^-_a = A^z + A^s$ with $A^s\leq 0$ and $A^z \leq 0$.  
Then $A$ is a product of two M-matrices thus $A^{-1}\geq 0$, if the following are satisfied:
\begin{enumerate}
\item $(A_d+A^z)\mathbf 1\neq \mathbf 0$  and $(A_d+A^z)\mathbf 1\geq 0$;
\item $A^+_a \leq A^zA^{-1}_dA^s$.
\end{enumerate}

\end{theorem}

\subsection{The fourth order   scheme in one dimension}

In the one dimension case, assume the domain $\Omega=[-L, L]$ is partitioned into $k$ uniform intervals with cell length $2h$. 
Then all $3$-point Gauss-Lobatto points for each small interval form 
an uniform grid  $-L=x_1<x_2<\cdots<x_N=L$ with grid spacing $h$ and $N=2k+1.$
  Thus the number of grid points for this fourth order scheme must be odd. 
 
 For convenience, we consider an equivalent form of \eqref{zxxeqn-fd4}:
 \begin{equation}
   W^{-1} S \tilde{\mathbf g}^{n+1}+ \frac{1}{\Delta t }M^n \tilde{\mathbf g}^{n+1}=  \frac{1}{\Delta t }M^n\mathbf g^{n} .
 \label{zxxeqn-fd5}
\end{equation}
 Let $A=  W^{-1} S+ \frac{1}{\Delta t }M^n$ and $\mathcal A: \mathbbm R^{N\times 1}\longrightarrow R^{N\times 1}$ be the scheme operator, i.e., 
  \eqref{zxxeqn-fd5} can be written as $\mathcal A(\tilde{\mathbf g}^{n+1})_i=\frac{1}{\Delta t }\mathcal M_i g^{n}_i.$   
Following the derivations in \cite{li2019monotonicity, li2020superconvergence}, with the same {\it ghost point values} notation in Section \ref{sec-1d-2ndscheme}, the finite element method with quadratic basis and 3-point Gauss Lobatto quadrature can be explicitly written as follows: for all $i=1,\cdots, N$, if $x_i$ is a cell end ($i$ is odd),  
   \begin{align}  
  \notag \mathcal A ( \tilde{\mathbf g}^{n+1})_i&:= \frac{(3 \mathcal M_{i-2}-4 \mathcal M_{i-1}+3 \mathcal M_i)\tilde{g}^{n+1}_{i-2}-(4 \mathcal M_{i-2}+12 \mathcal M_i)\tilde{g}^{n+1}_{i-1}}{8h^2}\\
  +& \frac{( \mathcal M_{i-2}+4 \mathcal M_{i-1}+18  \mathcal M_i+4 \mathcal M_{i+1}+ \mathcal M_{i+2})\tilde{g}^{n+1}_i}{8h^2}\nonumber\\
  \notag+&  \frac{-(12 \mathcal M_{i}+4 \mathcal M_{i+2})\tilde{g}^{n+1}_{i+1}+(3 \mathcal M_{i+2}-4 \mathcal M_{i+1}+3  \mathcal M_i)\tilde{g}^{n+1}_{i+2}}{8h^2}+\frac{\mathcal M_i}{\Delta t}\tilde{g}^{n+1}_i\\
  &=\frac{\mathcal M_i}{\Delta t}g_i^n;  \label{p2fd-vcoef-1dscheme-3} \end{align}
 and if $x_i$ is a cell center  ($i$ is even), 
   \begin{equation}
     \label{p2fd-vcoef-1dscheme-2}   
 \resizebox{\hsize}{!}{$
  \mathcal A ( \tilde{\mathbf g}^{n+1})_i:=\frac{-(3 \mathcal M_{i-1}+ \mathcal M_{i+1})\tilde{g}^{n+1}_{i-1}+4( \mathcal M_{i-1}+ \mathcal M_{i+1})\tilde{g}^{n+1}_i-( \mathcal M_{i-1}+3 \mathcal M_{i+1})\tilde{g}^{n+1}_{i+1}}{4h^2}+\frac{\mathcal M_i}{\Delta t} \tilde{g}^{n+1}_i=\frac{\mathcal M_i}{\Delta t} g_i^n. $}
 \end{equation} 
 
 Next, for the matrix $A$, we will discuss a decomposition of its negative off-diagonal parts of $A_a^{-}=A^z+A^s$ such that Theorem \ref{newthm3} can be verified under suitable mesh and time step constraints. 
 We will use operator notations to represent all matrices. 
 With the positive and negative parts for a number $f$ defined as: 
\[ f^+=\frac{|f|+f}{2},\quad f^-=\frac{|f|-f}{2},\]  
the linear operators $\mathcal A_d$, $\mathcal A_a^{\pm}$ are:
 \begin{align*}
 & \text{If $x_i$ is a cell end ($i$ is odd)},\hfill \\
&   \mathcal A_d(\tilde{\mathbf g}^{n+1})_i=
\left(\frac{\mathcal M_{i-2}+4\mathcal M_{i-1}+18 \mathcal M_i+4\mathcal M_{i+1}+\mathcal M_{i+2}}{8h^2}+\frac{\mathcal M_i}{\Delta t} \right)\tilde{g}^{n+1}_i;\\
& \resizebox{\hsize}{!}{$ \text{if $x_i$ is a cell center  ($i$ is even)},\quad 
  \mathcal A_d(\tilde{\mathbf g}^{n+1})_i=\left(\frac{\mathcal M_{i-1}+\mathcal M_{i+1}}{h^2}+\frac{\mathcal M_i}{\Delta t}\right) \tilde{g}^{n+1}_i. $}\\
&  \text{If $x_i$ is a cell end ($i$ is odd)},\\
&  \resizebox{\hsize}{!}{$ \mathcal A_a^+(\tilde{\mathbf g}^{n+1})_i=
\frac{(3\mathcal M_{i-2}-4\mathcal M_{i-1}+3\mathcal M_i)^+\tilde{g}^{n+1}_{i-2}+(3\mathcal M_{i+2}-4\mathcal M_{i+1}+3 \mathcal M_i)^+\tilde{g}^{n+1}_{i+2}}{8h^2}; $}\\
&  \text{if $x_i$ is a cell center  ($i$ is even)},\quad
  \mathcal A_a^+(\tilde{\mathbf g}^{n+1})_i=0.\\
 & \resizebox{\hsize}{!}{$ \text{If $x_i$ is a cell center},\quad \mathcal A_a^-(\tilde{\mathbf g}^{n+1})_i=\frac{-(3\mathcal M_{i-1}+\mathcal M_{i+1})\tilde{g}^{n+1}_{i-1}-(\mathcal M_{i-1}+3\mathcal M_{i+1})\tilde{g}^{n+1}_{i+1}}{4h^2}; $}\\
&  \text{if $x_i$ is a cell end},\quad  \mathcal A_a^-(\tilde{\mathbf g}^{n+1})_i=\frac{-(3\mathcal M_{i-2}-4\mathcal M_{i-1}+3\mathcal M_i)^-\tilde{g}^{n+1}_{i-2}}{8h^2}\\
&\resizebox{\hsize}{!}{$ +\frac{-(4\mathcal M_{i-2}+12\mathcal M_i)\tilde{g}^{n+1}_{i-1}-(12\mathcal M_{i}+4\mathcal M_{i+2})\tilde{g}^{n+1}_{i+1}-(3\mathcal M_{i}-4\mathcal M_{i+1}+3 \mathcal M_{i+2})^-\tilde{g}^{n+1}_{i+2}}{8h^2}.$}
  \end{align*}

  We can easily verify that $(A_d+A^z)\mathbf 1> 0$ for the following $\mathcal A^z$:
   \begin{align*}
&   \mbox{if $x_i$ is a cell center},\quad \mathcal A^z(\tilde{\mathbf g}^{n+1})_i=0, \\
&  \mbox{if $x_i$ is an interior cell end},\quad  \mathcal A^z(\tilde{\mathbf g}^{n+1})_i=\\
&\resizebox{\hsize}{!}{$ \frac{-(3\mathcal M_{i-2}-4\mathcal M_{i-1}+3\mathcal M_i)^-\tilde{g}^{n+1}_{i-2}-[4\mathcal M_{i-2}+12\mathcal M_i-(3\mathcal M_{i-2}-4\mathcal M_{i-1}+3\mathcal M_i)^+]\tilde{g}^{n+1}_{i-1}}{8h^2}$}\\
&\resizebox{\hsize}{!}{$ +\frac{-[12\mathcal M_{i}+4\mathcal M_{i+2}-(3\mathcal M_{i}-4\mathcal M_{i+1}+3 \mathcal M_{i+2})^+]\tilde{g}^{n+1}_{i+1}-(3\mathcal M_{i}-4\mathcal M_{i+1}+3 \mathcal M_{i+2})^-\tilde{g}^{n+1}_{i+2}}{8h^2}.$}
  \end{align*} 
  We can also verify that $A^s:=A^-_a-A^z\leq 0$:
     \begin{align*}
  & \resizebox{\hsize}{!}{$ \mbox{If $x_i$ is a cell center},\quad \mathcal A^s(\tilde{\mathbf g}^{n+1})_i= \frac{-(3\mathcal M_{i-1}+\mathcal M_{i+1})\tilde{g}^{n+1}_{i-1}-(\mathcal M_{i-1}+3\mathcal M_{i+1})\tilde{g}^{n+1}_{i+1}}{4h^2},$} \\
  &\mbox{If $x_i$ is a cell end},\\
  &\resizebox{\hsize}{!}{$   \mathcal A^s(\tilde{\mathbf g}^{n+1})_i=\frac{-(3\mathcal M_{i-2}-4\mathcal M_{i-1}+3\mathcal M_i)^+\tilde{g}^{n+1}_{i-1}-(3\mathcal M_{i}-4\mathcal M_{i+1}+3 \mathcal M_{i+2})^+ \tilde{g}^{n+1}_{i+1}}{8h^2}.$}
  \end{align*}
Now in order to verify $A^z A_d^{-1}A^s\geq A_a^+$ (entrywise inequality), 
 we only need to compare nonzero coefficients in 
  $\mathcal A^+_a(\tilde{\mathbf g}^{n+1})_i$ and $\mathcal A^z\left(\mathcal A^{-1}_d[\mathcal A^s(\tilde{\mathbf g}^{n+1})]\right)_i$ for $x_i$ being a  cell end.
    When $x_{i}$ is a cell end, $x_{i\pm1}$ are cell centers, and we have
  \[\mathcal A^s(\tilde{\mathbf g}^{n+1})_{i-1}= \frac{-(3\mathcal M_{i-2}+\mathcal M_{i})\tilde{g}^{n+1}_{i-2}-(\mathcal M_{i-2}+3\mathcal M_{i})\tilde{g}^{n+1}_{i}}{4h^2},\]
  \[\resizebox{\hsize}{!}{$  A^s(\tilde{\mathbf g}^{n+1})_{i-2}=\frac{-(3\mathcal M_{i-4}-4\mathcal M_{i-3}+3\mathcal M_{i-2})^+\tilde{g}^{n+1}_{i-3}-(3\mathcal M_{i-2}-4\mathcal M_{i-1}+3 \mathcal M_{i})^+ \tilde{g}^{n+1}_{i-1}}{8h^2},$}\]
\begin{equation*}
\resizebox{\hsize}{!}{$\mathcal A^{-1}_d[\mathcal A^s(\tilde{\mathbf g}^{n+1})]_{i-1}=\frac{h^2 \mathcal A^s(\tilde{\mathbf g}^{n+1})_{i-1}}{(\mathcal M_{i-2}+\mathcal M_i+h^2 \mathcal M_{i-1}/\Delta t)}= \frac{-(3\mathcal M_{i-2}+\mathcal M_{i})\tilde{g}^{n+1}_{i-2}-(\mathcal M_{i-2}+3\mathcal M_{i})\tilde{g}^{n+1}_{i}}{4(\mathcal M_{i-2}+\mathcal M_i+h^2 \mathcal M_{i-1}/\Delta t)}.$}
\end{equation*}
It suffices to focus on the coefficient of $\tilde{g}^{n+1}_{i-2}$ in  $\mathcal A^z(\mathcal A^{-1}_d[\mathcal A^s(\tilde{\mathbf g}^{n+1})])_i$ and the discussion for the coefficient of $\tilde{g}^{n+1}_{i+2}$ is similar. Notice that $\mathcal A^{-1}_d[\mathcal A^s(\tilde{\mathbf g}^{n+1})]_{i-2}$ will contribute nothing to the coefficient of $\tilde{g}^{n+1}_{i-2}$. So the coefficient of $\tilde{g}^{n+1}_{i-2}$ in  $\mathcal A^z(\mathcal A^{-1}_d[\mathcal A^s(\tilde{\mathbf g}^{n+1})])_i$ is 
\[  \frac{(3\mathcal M_{i-2}+\mathcal M_{i})(4\mathcal M_{i-2}+12\mathcal M_i-(3\mathcal M_{i-2}-4\mathcal M_{i-1}+3\mathcal M_i)^+)}{32h^2(\mathcal M_{i-2}+\mathcal M_i+h^2 \mathcal M_{i-1}/\Delta t)}.\]
Thus to ensure $A^+_a\leq A^zA_d^-A^s$, it suffices to have the following holds for any cell end $x_i$:
\begin{equation*}
\resizebox{.99\hsize}{!}{$
 \frac{(3\mathcal M_{i-2}+\mathcal M_{i})(4\mathcal M_{i-2}+12\mathcal M_i-(3\mathcal M_{i-2}-4\mathcal M_{i-1}+3\mathcal M_i)^+)}{32h^2(\mathcal M_{i-2}+\mathcal M_i+h^2 \mathcal M_{i-1}/\Delta t)}\geq \frac{(3\mathcal M_{i-2}-4\mathcal M_{i-1}+3\mathcal M_i)^+}{8h^2}.$}
\end{equation*}
Equivalently, we need the following inequality holds for any cell center $x_i$: 
\begin{equation}
\resizebox{.99\hsize}{!}{$
\frac{(3\mathcal M_{i-1}+\mathcal M_{i+1})(4\mathcal M_{i-1}+12\mathcal M_{i+1}-(3\mathcal M_{i-1}-4\mathcal M_{i}+3\mathcal M_{i+1})^+)}{32h^2(\mathcal M_{i-1}+\mathcal M_{i+1}+h^2  \mathcal M_{i}/\Delta t)}\geq \frac{(3\mathcal M_{i-1}-4\mathcal M_{i}+3\mathcal M_{i+1})^+}{8h^2}.$}\label{h-condition-1}
\end{equation}

If $3\mathcal M_{i-1}-4\mathcal M_{i}+3\mathcal M_{i+1}\leq 0$, then \eqref{h-condition-1} holds trivially. 
We only need to discuss the case $3\mathcal M_{i-1}-4\mathcal M_{i}+3\mathcal M_{i+1}>0$, for which \eqref{h-condition-1} becomes
\begin{equation}
\resizebox{.99\hsize}{!}{$
(3\mathcal M_{i-1}+\mathcal M_{i+1})(\mathcal M_{i-1}+4\mathcal M_{i}+9\mathcal M_{i+1})> 4(\mathcal M_{i-1}+\mathcal M_{i+1}+\frac{h^2}{\Delta t} \mathcal M_i)(3\mathcal M_{i-1}-4\mathcal M_{i}+3\mathcal M_{i+1}).$}
\label{h-condition-2}
\end{equation}

Let $a=\max\{\mathcal M_{i-1},\mathcal M_{i}, \mathcal M_{i+1} \}$ and  $b=\min\{\mathcal M_{i-1},\mathcal M_{i}, \mathcal M_{i+1} \}$, a convenient sufficient condition to ensure  \eqref{h-condition-2}  is 
$$56 b^2> 4\left(2+\frac{h^2}{\Delta t} \right) a(6a -4b),$$
which is equivalent to $2+\frac{h^2}{\Delta t}<14\frac{b^2}{6a^2-4ab}. $

 So we have proven the first result for the variable coefficient case:
\begin{theorem}
\label{1d-thm-mesh}
 For the scheme \eqref{zxxeqn-fd5} with   $\mathcal M_i> 0$, its matrix representation $A$ satisfies $A^{-1}\geq 0$ if  
 \eqref{h-condition-2} holds for any cell center $x_i$. A sufficient condition is to have the following constraints for each finite element cell $I_i=[x_{i-1}, x_{i+1}]$ ($i$ is even): 
 \begin{equation}
 \label{1D-cfl}
2+\frac{h^2}{\Delta t}<7\frac{1}{\max_{I_i} \mathcal M}\frac{\min_{I_i} \mathcal M^2}{3\max_{I_i} \mathcal M-2 \min_{I_i} \mathcal M},
 \end{equation}
 where $$\max_{I_i} \mathcal M:=\max\{\mathcal M_{i-1},\mathcal M_{i}, \mathcal M_{i+1} \},\quad \min_{I_i} \mathcal M:=\min\{\mathcal M_{i-1},\mathcal M_{i}, \mathcal M_{i+1} \}.$$
\end{theorem}    
 \begin{remark}
Note that for a smooth function $\mathcal M$, the mesh and time step constraints \eqref{1D-cfl} are possible to achieve because the right hand side of \eqref{1D-cfl} will converge to $7$ as $h$ goes to zero. Furthermore, for fixed $h$, the condition (\ref{1D-cfl}) gives a lower bound on $\Delta t$ (not an upper bound).
  \end{remark}

\subsection{The fourth order scheme in two dimensions}

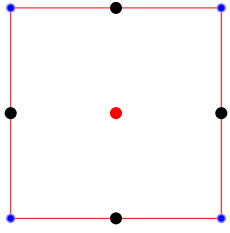
\begin{figure}
\begin{center}
 \scalebox{0.7}{
\begin{tikzpicture}[samples=100, domain=-3:3, place/.style={circle,draw=blue!50,fill=blue,thick,
inner sep=0pt,minimum size=1.5mm},transition/.style={circle,draw=red,fill=red,thick,inner sep=0pt,minimum size=2mm}
,point/.style={circle,draw=black,fill=black,thick,inner sep=0pt,minimum size=2mm}]

\draw[color=red] (-2,-2)--(-2,2);
\draw[color=red] (-2,-2)--(2,-2);
\draw[color=red] (2,-2)--(2, 2);
\draw[color=red] (-2,2)--(2, 2);
\node at ( -2,-2) [place] {};
\node at ( 0,-2) [point] {};
\node at ( 2,-2) [place] {};

\node at ( -2,0) [point] {};
\node at ( 0,0) [transition] {};
\node at ( 2,0) [point] {};

\node at ( -2,2) [place] {};
\node at ( 0,2) [point] {};
\node at ( 2,2) [place] {};
 \end{tikzpicture}
}
\caption{Three types of  grid points: red cell center, blue knots and black edge centers for a $Q^2$ finite element cell. }
\end{center}
\label{fig-points}
\end{figure}

Assume the domain is $\Omega=[-L, L]\times[-L,L]$ with 
an uniform $N\times N$ grid point with spacing $h$, obtained from all $3\times 3$ Gauss-Lobatto points on a uniform rectangular mesh with $k\times k$ cells. Thus $N=2k+1$.   
Let $\mathbf g$ be a $N\times N$ matrix with $g_{ij}$ denoting the point value at the $(i,j)$ grid point.  
For the $Q^2$ finite element method on uniform rectangular meshes, there are three types of grid point values, see Figure \ref{fig-points}.

 Let $A=  W^{-1} S+ \frac{1}{\Delta t }M^n$ and $\mathcal A: \mathbbm R^{N\times N}\longrightarrow R^{N\times N}$ be the scheme operator, i.e., 
  \eqref{zxxeqn-fd5} can be written as $\mathcal A(\tilde{\mathbf g}^{n+1})_{ij}=\frac{1}{\Delta t }\mathcal M_{ij} g^{n}_{ij}.$    
  With the same {\it ghost point values} notation as in Section \ref{sec-2D-2ndscheme},  following the derivations in \cite{li2019monotonicity}, the scheme can be explicitly written as:
 \begin{align*}
&  \resizebox{\hsize}{!}{$ \text{if $x_{ij}$ is a cell center},\quad \mathcal A_d (\tilde{\mathbf g}^{n+1})_{ij}=\left(\frac{\mathcal M_{i-1,j}+\mathcal M_{i+1,j}+\mathcal M_{i,j-1}+\mathcal M_{i,j+1}}{h^2}+\frac{1}{\Delta t}\mathcal M_{ij} \right)\tilde{g}^{n+1}_{ij}; $}\\
& \text{if $x_{ij}$ is an edge center for an edge parallel to $y$-axis},\\
& \resizebox{.99\hsize}{!}{$  \mathcal A_d (\tilde{\mathbf g}^{n+1})_{ij}=\left( \frac{(\mathcal M_{i-2,j}+4\mathcal M_{i-1,j}+18 \mathcal M_{ij}+4\mathcal M_{i+1,j}+\mathcal M_{i+2,j})+8(\mathcal M_{i,j-1}+\mathcal M_{i,j+1})}{8h^2}+\frac{1}{\Delta t}\mathcal M_{ij}\right)\tilde{g}^{n+1}_{ij};$}\\   
& \text{if $x_{ij}$ is an edge center for an edge parallel to $x$-axis},\\
&  \resizebox{.99\hsize}{!}{$ \mathcal A_d (\tilde{\mathbf g}^{n+1})_{ij}=\left( \frac{(\mathcal M_{i,j-2}+4\mathcal M_{i,j-1}+18 \mathcal M_{ij}+4\mathcal M_{i,j+1}+\mathcal M_{i,j+2})+8(\mathcal M_{i-1,j}+\mathcal M_{i+1,j})}{8h^2}+\frac{1}{\Delta t}\mathcal M_{ij}\right)\tilde{g}^{n+1}_{ij};$}\\
& \text{if $x_{ij}$ is a knot},\\
&\mathcal A_d (\tilde{\mathbf g}^{n+1})_{ij}= \left( \frac{\mathcal M_{i-2,j}+4\mathcal M_{i-1,j}+18 \mathcal M_{ij}+4\mathcal M_{i+1,j}+\mathcal M_{i+2,j}}{8h^2} \right.\\
& \left. + \frac{(\mathcal M_{i,j-2}+4\mathcal M_{i,j-1}+18 \mathcal M_{ij}+4\mathcal M_{i,j+1}+\mathcal M_{i,j+2})}{8h^2}+\frac{1}{\Delta t}\mathcal M_{ij}\right)\tilde{g}^{n+1}_{ij}.
\end{align*}
   
For the operator $\mathcal A^+_a$, it is given as
 \begin{align*}
& \text{if $x_{ij}$ is a cell center},\quad  \mathcal A_a^+ (\tilde{\mathbf g}^{n+1})_{ij}=0;\\
& \text{if $x_{ij}$ is an edge center for an edge parallel to $y$-axis},\\
&  \resizebox{\hsize}{!}{$  \mathcal A_a^+ (\tilde{\mathbf g}^{n+1})_{ij}=\frac{(3\mathcal M_{i-2,j}-4\mathcal M_{i-1,j}+3\mathcal M_{i,j})^+\tilde{g}^{n+1}_{i-2,j}+(3\mathcal M_{i+2,j}-4\mathcal M_{i+1,j}+3 \mathcal M_{i,j})^+\tilde{g}^{n+1}_{i+2,j}}{8h^2}; $}\\
& \text{if $x_{ij}$ is an edge center for an edge parallel to $x$-axis},\\
&  \resizebox{\hsize}{!}{$  \mathcal A_a^+ (\tilde{\mathbf g}^{n+1})_{ij}=\frac{(3\mathcal M_{i,j-2}-4\mathcal M_{i,j-1}+3\mathcal M_{i,j})^+\tilde{g}^{n+1}_{i,j-2}+(3\mathcal M_{i,j+2}-4\mathcal M_{i,j+1}+3 \mathcal M_{i,j})^+\tilde{g}^{n+1}_{i,j+2}}{8h^2};$}
\\
& \text{if $x_{ij}$ is a knot},\quad  \mathcal A_a^+ (\tilde{\mathbf g}^{n+1})_{ij}=\\
& \resizebox{\hsize}{!}{$  \frac{(3\mathcal M_{i-2,j}-4\mathcal M_{i-1,j}+3\mathcal M_{i,j})^+\tilde{g}^{n+1}_{i-2,j}+(3\mathcal M_{i+2,j}-4\mathcal M_{i+1,j}+3 \mathcal M_{i,j})^+\tilde{g}^{n+1}_{i+2,j}}{8h^2} $}\\
&\resizebox{\hsize}{!}{$ +\frac{(3\mathcal M_{i,j-2}-4\mathcal M_{i,j-1}+3\mathcal M_{i,j})^+\tilde{g}^{n+1}_{i,j-2}+(3\mathcal M_{i,j+2}-4\mathcal M_{i,j+1}+3 \mathcal M_{i,j})^+\tilde{g}^{n+1}_{i,j+2}}{8h^2}.$}
\end{align*}
We consider the following $A^z\leq 0$ and it is straightforward to see $(A_d+A^z)\mathbf 1> 0$:
    \begin{align*} 
    & \text{if $x_{ij}$ is a cell center}, \quad\mathcal A^z (\tilde{\mathbf g}^{n+1})_{ij}=0;\\
\\
& \text{if $x_{ij}$ is an edge center for an edge parallel to $y$-axis},\quad \mathcal A^z (\tilde{\mathbf g}^{n+1})_{ij}= \\
& \resizebox{.99\hsize}{!}{$\frac{-(3\mathcal M_{i-2,j}-4\mathcal M_{i-1,j}+3\mathcal M_{i,j})^-\tilde{g}^{n+1}_{i-2,j}-[4\mathcal M_{i-2,j}+12\mathcal M_{i,j}-(3\mathcal M_{i-2,j}-4\mathcal M_{i-1,j}+3\mathcal M_{i,j})^+]\tilde{g}^{n+1}_{i-1,j}}{8h^2}$} \\
&\resizebox{.99\hsize}{!}{$+\frac{-[12\mathcal M_{i,j}+4\mathcal M_{i+2,j}-(3\mathcal M_{i+2,j}-4\mathcal M_{i+1,j}+3 \mathcal M_{i,j})^+]\tilde{g}^{n+1}_{i+1,j}-(3\mathcal M_{i+2,j}-4\mathcal M_{i+1,j}+3 \mathcal M_{i,j})^-\tilde{g}^{n+1}_{i+2,j}}{8h^2};$}\\
& \mbox{if $x_{ij}$ is an edge center for an edge parallel to $x$-axis}, \quad \mathcal A^z (\tilde{\mathbf g}^{n+1})_{ij}= \\
& \resizebox{.99\hsize}{!}{$\frac{-(3\mathcal M_{i,j-2}-4\mathcal M_{i,j-1}+3\mathcal M_{i,j})^-\tilde{g}^{n+1}_{i,j-2}-[4\mathcal M_{i,j-2}+12\mathcal M_{i,j}-(3\mathcal M_{i,j-2}-4\mathcal M_{i,j-1}+3\mathcal M_{i,j})^+]\tilde{g}^{n+1}_{i,j-1}}{8h^2} $}\\
&\resizebox{.99\hsize}{!}{$ +\frac{-[12\mathcal M_{i,j}+4\mathcal M_{i,j+2}-(3\mathcal M_{i,j+2}-4\mathcal M_{i,j+1}+3 \mathcal M_{i,j})^+]\tilde{g}^{n+1}_{i,j+1}-(3\mathcal M_{i,j+2}-4\mathcal M_{i,j+1}+3 \mathcal M_{i,j})^-\tilde{g}^{n+1}_{i,j+2}}{8h^2};$}\\
& \mbox{if $x_{ij}$ is a knot},\quad\mathcal A^z (\tilde{\mathbf g}^{n+1})_{ij}= \\
& \resizebox{.99\hsize}{!}{$\frac{-(3\mathcal M_{i-2,j}-4\mathcal M_{i-1,j}+3\mathcal M_{i,j})^-\tilde{g}^{n+1}_{i-2,j}-[4\mathcal M_{i-2,j}+12\mathcal M_{i,j}-(3\mathcal M_{i-2,j}-4\mathcal M_{i-1,j}+3\mathcal M_{i,j})^+]\tilde{g}^{n+1}_{i-1,j}}{8h^2} $}\\
&\resizebox{.99\hsize}{!}{$+\frac{-[12\mathcal M_{i,j}+4\mathcal M_{i+2,j}-(3\mathcal M_{i+2,j}-4\mathcal M_{i+1,j}+3 \mathcal M_{i,j})^+]\tilde{g}^{n+1}_{i+1,j}-(3\mathcal M_{i+2,j}-4\mathcal M_{i+1,j}+3 \mathcal M_{i,j})^-\tilde{g}^{n+1}_{i+2,j}}{8h^2}$}\\
&\resizebox{.99\hsize}{!}{$+\frac{-(3\mathcal M_{i,j-2}-4\mathcal M_{i,j-1}+3\mathcal M_{i,j})^-\tilde{g}^{n+1}_{i,j-2}-[4\mathcal M_{i,j-2}+12\mathcal M_{i,j}-(3\mathcal M_{i,j-2}-4\mathcal M_{i,j-1}+3\mathcal M_{i,j})^+]\tilde{g}^{n+1}_{i,j-1}}{8h^2}$} \\
&\resizebox{.99\hsize}{!}{$+\frac{-[12\mathcal M_{i,j}+4\mathcal M_{i,j+2}-(3\mathcal M_{i,j+2}-4\mathcal M_{i,j+1}+3 \mathcal M_{i,j})^+]\tilde{g}^{n+1}_{i,j+1}-(3\mathcal M_{i,j+2}-4\mathcal M_{i,j+1}+3 \mathcal M_{i,j})^-\tilde{g}^{n+1}_{i,j+2}}{8h^2}$};
\end{align*}
Then $A^s=A_a^- -A^z$ is given as:
\begin{align*} 
&\resizebox{\hsize}{!}{$ \mbox{if $x_i$ is a cell center}, \quad \mathcal A^s (\tilde{\mathbf g}^{n+1})_{ij}=-\frac{(3\mathcal M_{i-1,j}+\mathcal M_{i+1,j})\tilde{g}^{n+1}_{i-1,j}+(\mathcal M_{i-1,j}+3\mathcal M_{i+1,j})\tilde{g}^{n+1}_{i+1,j}}{4h^2}$}\\
&-\frac{(3\mathcal M_{i,j-1}+\mathcal M_{i,j+1})\tilde{g}^{n+1}_{i,j-1}+(\mathcal M_{i,j-1}+3\mathcal M_{i,j+1})\tilde{g}^{n+1}_{i,j+1}}{4h^2};\\
& \mbox{if $x_{ij}$ is an edge center for an edge parallel to $y$-axis},  \quad \mathcal A^s (\tilde{\mathbf g}^{n+1})_{ij}=\\
&  \resizebox{\hsize}{!}{$ \frac{-(3\mathcal M_{i-2,j}-4\mathcal M_{i-1,j}+3\mathcal M_{i,j})^+\tilde{g}^{n+1}_{i-1,j}-(3\mathcal M_{i+2,j}-4\mathcal M_{i+1,j}+3 \mathcal M_{i,j})^+\tilde{g}^{n+1}_{i+1,j}}{8h^2}$}\\
&+\frac{-(3\mathcal M_{i,j-1}+\mathcal M_{i,j+1})\tilde{g}^{n+1}_{i,j-1}-(\mathcal M_{i,j-1}+3\mathcal M_{i,j+1})\tilde{g}^{n+1}_{i,j+1}}{4h^2};\end{align*}
\begin{align*} 
& \mbox{if $x_{ij}$ is an edge center for an edge parallel to $x$-axis},\quad \mathcal A^s (\tilde{\mathbf g}^{n+1})_{ij}=\\
&\resizebox{\hsize}{!}{$  \frac{-(3\mathcal M_{i,j-2}-4\mathcal M_{i,j-1}+3\mathcal M_{i,j})^+\tilde{g}^{n+1}_{i,j-1}-(3\mathcal M_{i,j+2}-4\mathcal M_{i,j+1}+3 \mathcal M_{i,j})^+\tilde{g}^{n+1}_{i,j+1}}{8h^2}$}\\
&+\frac{-(3\mathcal M_{i-1,j}+\mathcal M_{i+1,j})\tilde{g}^{n+1}_{i-1,j}-(\mathcal M_{i-1,j}+3\mathcal M_{i+1,j})\tilde{g}^{n+1}_{i+1,j}}{4h^2};\\
& \mbox{if $x_{ij}$ is a knot},\quad\mathcal A^s (\tilde{\mathbf g}^{n+1})_{ij}=\\
&\resizebox{\hsize}{!}{$ \frac{-(3\mathcal M_{i-2,j}-4\mathcal M_{i-1,j}+3\mathcal M_{i,j})^+\tilde{g}^{n+1}_{i-1,j}-(3\mathcal M_{i+2,j}-4\mathcal M_{i+1,j}+3 \mathcal M_{i,j})^+\tilde{g}^{n+1}_{i+1,j}}{8h^2}$} \\
&\resizebox{\hsize}{!}{$ +\frac{-(3\mathcal M_{i,j-2}-4\mathcal M_{i,j-1}+3\mathcal M_{i,j})^+\tilde{g}^{n+1}_{i,j-1}-(3\mathcal M_{i,j+2}-4\mathcal M_{i,j+1}+3 \mathcal M_{i,j})^+\tilde{g}^{n+1}_{i,j+1}}{8h^2}.$}
\end{align*}

For the positive off-diagonal entries, $\mathcal A^+_a (\tilde{\mathbf g}^{n+1})_{ij}$ is nonzero only for $x_{ij}$ being an edge center or a cell center.
Thus to verify  $A_a^{+}\leq A^zA_d^{-1}A^s$, it suffices to compare $\mathcal A^z\left[\mathcal A_d^{-1}\left(\mathcal A^s(\tilde{\mathbf g}^{n+1})\right)\right]_{ij}$ with 
$\mathcal A_a^{+}(\tilde{\mathbf g}^{n+1})_{ij}$ for $x_{ij}$ being an edge center or a cell center.

If $x_{ij}$ is an edge center for an edge parallel to $y$-axis, then $x_{i\pm1,j}$ are cell centers. Since everything here has a symmetric structure, we only need to compare the coefficients of
$\tilde{g}^{n+1}_{i-2,j}$ in $\mathcal A^z\left[\mathcal A_d^{-1}\left(\mathcal A^s(\tilde{\mathbf g}^{n+1})\right)\right]_{ij}$ and
$\mathcal A_a^{+}(\tilde{\mathbf g}^{n+1})_{ij}$, and the comparison for the coefficients of
$\tilde{g}^{n+1}_{i+2,j}$ will be similar. 
\begin{align*} 
&\mathcal A^s (\tilde{\mathbf g}^{n+1})_{i-1,j} =-\frac{(3\mathcal M_{i-2,j}+\mathcal M_{ij})\tilde{g}^{n+1}_{i-2,j}+(\mathcal M_{i-2,j}+3\mathcal M_{i,j})\tilde{g}^{n+1}_{i,j}}{4h^2}\\
- & \frac{(3\mathcal M_{i-1,j-1}+\mathcal M_{i-1,j+1})\tilde{g}^{n+1}_{i-1,j-1}+(\mathcal M_{i-1,j-1}+3\mathcal M_{i-1,j+1})\tilde{g}^{n+1}_{i-1,j+1}}{4h^2},\\
 &\resizebox{\hsize}{!}{$  \mathcal A_d^{-1}[\mathcal A^s (\tilde{\mathbf g}^{n+1})]_{i-1,j}  =- \frac{(3\mathcal M_{i-2,j}+\mathcal M_{ij})\tilde{g}^{n+1}_{i-2,j}+(\mathcal M_{i-2,j}+3\mathcal M_{ij})\tilde{g}^{n+1}_{i,j}}{4(\mathcal M_{i-2,j}+\mathcal M_{ij}+\mathcal M_{i-1,j+1}+\mathcal M_{i-1,j-1}+h^2\frac{1}{\Delta t}\mathcal M_{i-1,j})}$}\\
&\resizebox{\hsize}{!}{$ - \frac{(3\mathcal M_{i-1,j-1}+\mathcal M_{i-1,j+1})\tilde{g}^{n+1}_{i-1,j-1}+(\mathcal M_{i-1,j-1}+3\mathcal M_{i-1,j+1})\tilde{g}^{n+1}_{i-1,j+1}}{4(\mathcal M_{i-2,j}+\mathcal M_{ij}+\mathcal M_{i-1,j+1}+\mathcal M_{i-1,j-1}+h^2\frac{1}{\Delta t}\mathcal M_{i-1,j})}.$}
 \end{align*}
 
 Since the coefficient of
$\tilde{g}^{n+1}_{i-2,j}$ in $\mathcal A_a^+(\tilde{\mathbf g}^{n+1})_{ij}$ is $(3\mathcal M_{i-2,j}-4 \mathcal M_{i-1,j}+3\mathcal M_{ij})^+/(8h^2)$, we only need to discuss the case $3\mathcal M_{i-2,j}-4 \mathcal M_{i-1,j}+3\mathcal M_{ij}>0$, for which the coefficient of
$\tilde{g}^{n+1}_{i-2,j}$ in $\mathcal A^z\left[\mathcal A_d^{-1}\left(\mathcal A^s(\tilde{\mathbf g}^{n+1})\right)\right]_{ij}$ becomes 
\[ \resizebox{\hsize}{!}{$  \frac{\mathcal M_{i-2,j}+4 \mathcal M_{i-1,j}+9\mathcal M_{ij}}{8 h^2}\frac{ (3\mathcal M_{i-2,j}+\mathcal M_{ij})}{4(\mathcal M_{i-2,j}+\mathcal M_{ij}+\mathcal M_{i-1,j+1}+\mathcal M_{i-1,j-1}+h^2\frac{1}{\Delta t}\mathcal M_{i-1,j})}.$} \]
To ensure the coefficient of
$\tilde{g}^{n+1}_{i-2,j}$ in $\mathcal A^z\left[\mathcal A_d^{-1}\left(\mathcal A^s(\tilde{\mathbf g}^{n+1})\right)\right]_{ij}$ is no less than the coefficient of
$\tilde{g}^{n+1}_{i-2,j}$ in $\mathcal A_a^{+}(\tilde{\mathbf g}^{n+1})_{ij}$, we need  
\[\resizebox{\hsize}{!}{$  \frac{ (\mathcal M_{i-2,j}+4 \mathcal M_{i-1,j}+9\mathcal M_{ij})(3\mathcal M_{i-2,j}+\mathcal M_{ij})}{32 h^2(\mathcal M_{i-2,j}+\mathcal M_{ij}+\mathcal M_{i-1,j+1}+\mathcal M_{i-1,j-1}+h^2\frac{1}{\Delta t}\mathcal M_{i-1,j})}\geq \frac{3 \mathcal M_{i-2,j}-4\mathcal M_{i-1,j}+3\mathcal M_{ij}}{8 h^2}.$}\]
Similar to the one-dimensional case, it suffices to require
\[\resizebox{\hsize}{!}{$ \frac{(\mathcal M_{i-2,j}+4 \mathcal M_{i-1,j}+9\mathcal M_{ij})(3\mathcal M_{i-2,j}+\mathcal M_{ij})}{4 (\mathcal M_{i-2,j}+\mathcal M_{ij}+\mathcal M_{i-1,j+1}+\mathcal M_{i-1,j-1}+h^2\frac{1}{\Delta t}\mathcal M_{i-1,j})}>3 \mathcal M_{i-2,j}-4\mathcal M_{i-1,j}+3\mathcal M_{ij}.$}\]
Equivalently, we need the following inequality holds for any cell center $x_{ij}$: 
\begin{subequations}
 \label{h-condition-2d-cellcenter}
 \begin{align}
\label{h-condition-2d-cellcenter-1}
\resizebox{\hsize}{!}{$  \frac{(\mathcal M_{i-1,j}+4 \mathcal M_{i,j}+9\mathcal M_{i+1,j})(3\mathcal M_{i-1,j}+\mathcal M_{i+1,j})}{4 (\mathcal M_{i-1,j}+\mathcal M_{i+1,j}+\mathcal M_{i,j+1}+\mathcal M_{i,j-1}+h^2\frac{1}{\Delta t}\mathcal M_{i,j})}>3 \mathcal M_{i-1,j}-4\mathcal M_{i,j}+3\mathcal M_{i+1,j}.$}
\end{align}
Notice that \eqref{h-condition-2d-cellcenter-1} was derived for comparing $\mathcal A^z\left[\mathcal A_d^{-1}\left(\mathcal A^s(\tilde{\mathbf g}^{n+1})\right)\right]_{ij}$ and $\mathcal A_a^{+}(\tilde{\mathbf g}^{n+1})_{ij}$ for $x_{ij}$ being an edge center of an edge parallel to $y$-axis. 
If $x_{ij}$ is an edge center of an edge parallel to $x$-axis, then we can derive a similar constraint:
 \begin{align}
\label{h-condition-2d-cellcenter-2}
 \resizebox{\hsize}{!}{$  \frac{(\mathcal M_{i,j-1}+4 \mathcal M_{i,j}+9\mathcal M_{i,j+1})(3\mathcal M_{i,j-1}+\mathcal M_{i,j+1})}{4 (\mathcal M_{i,j-1}+\mathcal M_{i,j+1}+\mathcal M_{i+1,j}+\mathcal M_{i-1,j}+h^2\frac{1}{\Delta t}\mathcal M_{i,j})}>3 \mathcal M_{i,j-1}-4\mathcal M_{i,j}+3\mathcal M_{i,j+1}.$}
\end{align}
\end{subequations}

If $x_{ij}$ is a knot, then $x_{i\pm1,j}$ are edge centers for an edge parallel to $x$-axis. Since everything here has a symmetric structure, we only need to compare the coefficients of
$\tilde{g}^{n+1}_{i-2,j}$ in $\mathcal A^z\left[\mathcal A_d^{-1}\left(\mathcal A^s(\tilde{\mathbf g}^{n+1})\right)\right]_{ij}$ and
$\mathcal A_a^{+}(\tilde{\mathbf g}^{n+1})_{ij}$, and the comparison for the coefficients of
$\tilde{g}^{n+1}_{i+2,j}$, $\tilde{g}^{n+1}_{i,j-2}$ and $\tilde{g}^{n+1}_{i,j+2}$ will be similar. 
\begin{align*}
 &\mathcal A^s (\tilde{\mathbf g}^{n+1})_{i-1,j}=\resizebox{.55\hsize}{!}{$\frac{-(3\mathcal M_{i-2,j}+\mathcal M_{i,j})\tilde{g}^{n+1}_{i-2,j}-(\mathcal M_{i-2,j}+3\mathcal M_{i,j})\tilde{g}^{n+1}_{i,j}}{4h^2}$}\\
&\resizebox{.99\hsize}{!}{$+\frac{-(3\mathcal M_{i-1,j-2}-4\mathcal M_{i-1,j-1}+3\mathcal M_{i-1,j})^+\tilde{g}^{n+1}_{i-1,j-1}-(3\mathcal M_{i-1,j+2}-4\mathcal M_{i-1,j+1}+3 \mathcal M_{i-1,j})^+\tilde{g}^{n+1}_{i-1,j+1}}{8h^2}$}
\\
 &\mathcal A_d^{-1}[\mathcal A^s (\tilde{\mathbf g}^{n+1})]_{i-1,j}=\\
&\resizebox{.99\hsize}{!}{$\frac{-(3\mathcal M_{i-2,j}+\mathcal M_{i,j})\tilde{g}^{n+1}_{i-2,j}-(\mathcal M_{i-2,j}+3\mathcal M_{i,j})\tilde{g}^{n+1}_{i,j}}{\frac{1}{2}(\mathcal M_{i-1,j-2}+4\mathcal M_{i-1,j-1}+18 \mathcal M_{i-1,j}+4\mathcal M_{i-1,j+1}+\mathcal M_{i-1,j+2})+4(\mathcal M_{i-2,j}+\mathcal M_{i,j})+4h^2\frac{1}{\Delta t}\mathcal M_{i-1,j}}$}\\
&\resizebox{.99\hsize}{!}{$+\frac{-(3\mathcal M_{i-1,j-2}-4\mathcal M_{i-1,j-1}+3\mathcal M_{i-1,j})^+\tilde{g}^{n+1}_{i-1,j-1}-(3\mathcal M_{i-1,j+2}-4\mathcal M_{i-1,j+1}+3 \mathcal M_{i-1,j})^+\tilde{g}^{n+1}_{i-1,j+1}}{(\mathcal M_{i-1,j-2}+4\mathcal M_{i-1,j-1}+18 \mathcal M_{i-1,j}+4\mathcal M_{i-1,j+1}+\mathcal M_{i-1,j+2})+8(\mathcal M_{i-2,j}+\mathcal M_{i,j})+8h^2\frac{1}{\Delta t}\mathcal M_{i-1,j}}.$}
 \end{align*}
 For the same reason as above we still only consider the case where $3\mathcal M_{i-2,j}-4 \mathcal M_{i-1,j}+3\mathcal M_{ij}>0$. So the coefficient of
$\tilde{g}^{n+1}_{i-2,j}$ in $\mathcal A^z\left[\mathcal A_d^{-1}\left(\mathcal A^s(\tilde{\mathbf g}^{n+1})\right)\right]_{ij}$ is 
\[\resizebox{.99\hsize}{!}{$ \frac{1}{4 h^2}\frac{(\mathcal M_{i-2,j}+4 \mathcal M_{i-1,j}+9\mathcal M_{ij})(3\mathcal M_{i-2,j}+\mathcal M_{i,j})}{(\mathcal M_{i-1,j-2}+4\mathcal M_{i-1,j-1}+18 \mathcal M_{i-1,j}+4\mathcal M_{i-1,j+1}+\mathcal M_{i-1,j+2})+8(\mathcal M_{i-2,j}+\mathcal M_{i,j})+8\frac{1}{\Delta t}\mathcal M_{i-1,j}h^2}.$} \]
To ensure the coefficient of
$\tilde{g}^{n+1}_{i-2,j}$ in $\mathcal A^z\left[\mathcal A_d^{-1}\left(\mathcal A^s(\tilde{\mathbf g}^{n+1})\right)\right]_{ij}$ is no less than the coefficient of
$\tilde{g}^{n+1}_{i-2,j}$ in $\mathcal A_a^{+}(\tilde{\mathbf g}^{n+1})_{ij}$,  we only need  
\[\resizebox{.99\hsize}{!}{$\frac{2(\mathcal M_{i-2,j}+4 \mathcal M_{i-1,j}+9\mathcal M_{ij})(3\mathcal M_{i-2,j}+\mathcal M_{i,j})}{(\mathcal M_{i-1,j-2}+4\mathcal M_{i-1,j-1}+18 \mathcal M_{i-1,j}+4\mathcal M_{i-1,j+1}+\mathcal M_{i-1,j+2})+8(\mathcal M_{i-2,j}+\mathcal M_{i,j})+8\frac{1}{\Delta t}\mathcal M_{i-1,j}h^2}$}\]
\[> 3 \mathcal M_{i-2,j}-4\mathcal M_{i-1,j}+3\mathcal M_{ij}.\]

Equivalently, we need the following inequality holds for any edge center  $x_{ij}$  for an edge parallel to $x$-axis: 
\begin{subequations}
 \label{h-condition-2d-edgecenter}
\begin{align}\label{h-condition-2d-xedgecenter}
\resizebox{\hsize}{!}{$ \frac{2(\mathcal M_{i-1,j}+4 \mathcal M_{i,j}+9\mathcal M_{i+1,j})(3\mathcal M_{i-1,j}+\mathcal M_{i+1,j})}{(\mathcal M_{i,j-2}+4\mathcal M_{i,j-1}+18 \mathcal M_{i,j}+4\mathcal M_{i,j+1}+\mathcal M_{i,j+2})+8(\mathcal M_{i-1,j}+\mathcal M_{i+1,j})+8c_{i,j}h^2}$}\nonumber\\
> 3 \mathcal M_{i-1,j}-4\mathcal M_{i,j}+3\mathcal M_{i+1,j}.
\end{align}
We also need the following inequality holds for any edge center  $x_{ij}$  for an edge parallel to $y$-axis: 
\begin{align}\label{h-condition-2d-yedgecenter}
\resizebox{\hsize}{!}{$ \frac{2(\mathcal M_{i,j-1}+4 \mathcal M_{i,j}+9\mathcal M_{i,j+1})(3\mathcal M_{i,j-1}+\mathcal M_{i,j-1})}{(\mathcal M_{i-2,j}+4\mathcal M_{i-1,j}+18 \mathcal M_{i,j}+4\mathcal M_{i+1,j}+\mathcal M_{i+2,j})+8(\mathcal M_{i,j-1}+\mathcal M_{i,j+1})+8c_{i,j}h^2}$}\nonumber\\
> 3 \mathcal M_{i,j-1}-4\mathcal M_{i,j}+3\mathcal M_{i,j+1}.
\end{align}
\end{subequations}
We have similar result to the one-dimensional case as following:
\begin{theorem}
 For the scheme \eqref{zxxeqn-fd5}, its matrix representation $A$ satisfies $A^{-1}\geq 0$ if \eqref{h-condition-2d-cellcenter} holds for any cell center $x_{ij}$, \eqref{h-condition-2d-xedgecenter} holds for $x_{ij}$ being  any edge center of an edge parallel to $x$-axis  and   \eqref{h-condition-2d-yedgecenter} holds for  $x_{ij}$ being any edge center of an edge parallel to $y$-axis.
\end{theorem}
\begin{theorem}\label{2d-thm-mesh-1}
For the scheme \eqref{zxxeqn-fd5}, its matrix representation $A$ satisfies $A^{-1}\geq 0$ if  
the following mesh constraint is achieved for all edge centers $x_{ij}$:
\begin{equation}
\label{2d-constraint}
\frac{11}{2}+\frac{h^2}{\Delta t}<7\frac{1}{\max_{J_{ij}} \mathcal M}\frac{\min_{J_{ij}} \mathcal M^2}{3\max_{J_{ij}} \mathcal M-2 \min_{J_{ij}} \mathcal M},\end{equation}
where $J_{ij}$ is the union of two finite element cells: if $x_{ij}$ is an edge center of an edge parallel to $x$-axis, then  $J_{ij}= [x_{i-1}, x_{i+1}]\times[y_{j-2}, y_{j+2}]$; if $x_{ij}$ is an edge center of an edge parallel to $y$-axis, then  $J_{ij}= [x_{i-2}, x_{i+2}]\times[y_{j-1}, y_{j+1}]$. Here the maximum and minimum of $\mathcal M$ are those of grid point values of $\mathcal M$ in $J_{ij}$.
\end{theorem}

\begin{remark}
Similarly as the one dimensional case, for smooth $\mathcal{M}$, the constraint \eqref{2d-constraint} can be satisfied for small $h$. 
\end{remark}

\section{Positivity and energy dissipation}
\label{sec-positivity}
In this section, we prove a few properties of the proposed scheme (\ref{zxxeqn-fd3}), among which positivity and energy dissipation are the most important ones. First of all, we rewrite \eqref{zxxeqn-fd3} as 
\begin{equation} \label{scheme11}
A \tilde{\mathbf g}^{n+1}=\mathbf g^{n}, \quad A:=I+ \Delta t  (M^n)^{-1}W^{-1}S.
\end{equation}
From the previous section, we know that the matrix $A$ is invertible and $A^{-1}\geq 0$ under suitable mesh size and time step constraints. Specifically, the second order scheme is always monotone $A^{-1}\geq 0$ (entrywise inequality) for any mesh  size and time step. For the fourth order scheme, assume that the mesh size and time step satisfy the constraints \eqref{1D-cfl} and \eqref{2d-constraint} in one and two dimensions, respectively, we also have  $A^{-1}\geq 0$.

\subsection{Conservation, steady state and positivity}
It is straightforward to verify the following properties:
\begin{enumerate}
\item {\it Mass conservation of $\rho$}. Multiplying ${\bf 1}^T WM^n$ from the left on both sides of (\ref{scheme11}) and using ${\bf 1}^TS=\mathbf 0^T$ gives
$${\bf 1}^T WM^n {\bf \tilde g}^{n+1}={\bf 1}^T WM^n \mathbf g^{n},$$
which is
\[\mathbf 1^T W\boldsymbol{\rho}^{n+1}=\mathbf 1^T W\boldsymbol{\rho}^{n}, \]
or equivalently,
\[\sum_i w_i \rho_i^{n+1}=\sum_i w_i \rho_i^{n}.\]
\item   {\it Mass conservation of $c$}. By setting $v_h\equiv 1$ in \eqref{scheme-c-n}, we get 
$\alpha\langle c_h^n, 1\rangle=\langle \rho_h^n, 1\rangle$ thus
\[ \alpha \sum_i w_i c_i^{n}=\sum_i w_i \rho_i^{n}. \]
\item {\it Steady state preserving}. If ${\bf g}^n=C{\bf 1}$ for some constant $C$, then using $S{\bf 1}={\bf 0}$ it can be easily seen that ${\bf \tilde g}^{n+1}=C{\bf 1}$ is the unique solution to (\ref{scheme11}). In terms of the $\rho$ variable, this implies that  
$$\rho^n_i=C\mM_i^n, \forall i\Longrightarrow \rho^{n+1}_i=C\mM_i^n,\forall i.$$

 \item {\it Positivity of $\rho$}.  If $\rho^n_i>0$ for every $i$, then $g_i^n=\rho_i^n/\mM_i^n>0$ for every $i$. When $A^{-1}\geq 0$ holds, we have $\tilde{g}^{n+1}_i>0$, consequently $\rho^{n+1}_i=\mM_i^n \tilde g_i^{n+1}>0$ for every $i$.  
 
\item {\it Positivity of $c$}. 
All discussion in Section \ref{sec-mono} applies to the scheme \eqref{scheme-c-n} with $\alpha>0$ and  suitable boundary conditions. Even though we only consider Neumann type boundary condition in this paper, the results hold also for Dirichlet type boundary conditions. In particular, the second order scheme is monotone. By setting $\mathcal M\equiv 1$ and $\Delta t=\frac{1}{\alpha}$ in Theorem \ref{1d-thm-mesh} and Theorem \ref{2d-thm-mesh-1}, the fourth order scheme is also monotone if $\alpha h^2\leq 5$ in one dimension and $\alpha h^2\leq \frac32$ in two dimensions. When monotonicity in  \eqref{scheme-c-n}  holds, positivity of $c$ is implied by positivity of $\rho$. 
\end{enumerate}

\subsection{Energy dissipation}

In this subsection, we show that the fully discrete scheme (\ref{scheme11}) decays energy. Following the continuous counterpart (\ref{energy1}), we define the discrete energy as
\begin{equation}
E^n:=  \left\langle \rho^n \log \frac{\rho^n}{\mathcal M^n}-\rho^n+\frac12c^n\rho^n, 1\right\rangle=\sum_i w_i \left( \rho^n_i \log \frac{\rho_i^n}{\mathcal M_i^n}-\rho^n_i+\frac12c^n_i\rho_i^n\right).
\label{energy-definition}
\end{equation}
Note that by using $c_i^n$ we consider the Keller-Segel equation directly. In the Fokker-Planck case, the last term $\frac12c^n_i\rho_i^n$ in $E^n$ is zero.

\begin{theorem}
Assume monotonicity holds for scheme \eqref{scheme11}, i.e.,  $A^{-1}\geq 0$, for the energy defined in (\ref{energy-definition}) we have $E^{n+1}\leq E^n.$
\end{theorem}
\begin{proof}

First of all,
\begin{equation*}\resizebox{\hsize}{!}{$ 
\begin{split}
&E^{n+1}-E^n\\
=&\sum_i w_i \left( \rho^{n+1}_i \log \frac{\rho_i^{n+1}}{\mathcal M_i^{n+1}}-\rho^{n+1}_i+\frac12c^{n+1}_i\rho_i^{n+1}\right)-\sum_i w_i \left( \rho^n_i \log \frac{\rho_i^n}{\mathcal M_i^n}-\rho^n_i+\frac12c^n_i\rho_i^n\right)\\
=&\sum_i w_i \left( \rho^{n+1}_i \log \frac{\rho_i^{n+1}}{\mathcal M_i^{n+1}}+\frac12c^{n+1}_i\rho_i^{n+1}\right)-\sum_i w_i \left( \rho^n_i \log \frac{\rho_i^n}{\mathcal M_i^n}+\frac12c^n_i\rho_i^n\right)\\
=&I+II,
\end{split}$}
\end{equation*}
where we used mass conservation in the second equality and 
\begin{equation*}
\begin{split}
&I:=\sum_i w_i \rho^{n+1}_i \log \frac{\rho^{n+1}_i}{\mathcal M^n_i}-   \sum_i w_i \rho^{n}_i \log \frac{\rho^{n}_i}{\mathcal M^n_i },\\
&II:=\sum_i w_i \left(\rho_i^{n+1}c_i^n-\frac{1}{2}\rho_i^{n+1}c_i^{n+1}-\frac{1}{2} \rho_i^nc_i^n\right).
\end{split}
\end{equation*}

On the other hand, it is easy to see $A^{-1}{\bf 1}={\bf 1}$, since $A{\bf 1}={\bf 1}$. Let $a^{ij}$ be the entries of $A^{-1}$, then $\sum_{j}a^{ij}=1$ and $a^{ij}\geq 0$ for all $i$, $j$ if the monotonicity holds. Furthermore, since $M^n$ and
$W$ are diagonal matrices, $M^nW=WM^n$ thus ${\bf 1}^TM^nW A= \mathbf 1^T M^nW (I+\Delta t (M^n)^{-1}W^{-1}S)={\bf 1}^TM^nW$. So we have ${\bf 1}^TM^nWA^{-1}={\bf 1}^TM^nW$,
which is $\sum_{i}\mathcal M^n_i w_i  a^{ij} =\mathcal M^n_j w_j$ componentwise.
 
The above discussion implies that $\tilde{g}^{n+1}_i=\sum_{j} a^{ij}g^n_j$ is a convex combination. The function $x\log x$ is convex, so by Jensen's inequality, 
 \begin{equation*}
 \tilde{g}^{n+1}_i \log (\tilde{g}^{n+1}_i)\leq \sum_{j} a^{ij}g^n_j\log (g^n_j).
 \label{emp}
 \end{equation*}
Then 
 \[\resizebox{\hsize}{!}{$  \sum_i w_i \rho^{n+1}_i \log (\rho^{n+1}_i/\mathcal M^n_i )= \sum_i w_i \mathcal M^n_i \tilde{g}^{n+1}_i \log (\tilde{g}^{n+1}_i)\leq \sum_i w_i \mathcal M^n_i \sum_{j} a^{ij}g^n_j\log (g^n_j) $} \]
 \[\resizebox{\hsize}{!}{$  =\sum_j\left (\sum_i a^{ij}w_i \mathcal M^n_i \right) g^n_j\log (g^n_j)=\sum_j w_j \mathcal M^n_j  g^n_j\log (g^n_j)=\sum_i w_i \rho^{n}_i \log (\rho^{n}_i/\mathcal M^n_i ). $}\]
We thus proved $I\leq 0$. The proof is done if it is the Fokker-Planck equation.

If it is the Keller-Segel equation, we still need to show $II\leq 0$. Recall that we use the scheme (\ref{scheme-c-n}) for $c$:
\begin{equation} \label{schemeC1}
 \langle \nabla c_h^n, \nabla v_h \rangle +\alpha\langle  c_h^n,  v_h \rangle= \langle  \rho^n,  v_h \rangle,\quad \forall v_h\in V^h.
\end{equation}
At $t^{n+1}$, this is
\begin{equation}  \label{schemeC2}
 \langle \nabla c_h^{n+1}, \nabla v_h \rangle +\alpha\langle  c_h^{n+1},  v_h \rangle= \langle  \rho^{n+1},  v_h \rangle,\quad \forall v_h\in V^h.
\end{equation}
Subtracting (\ref{schemeC1}) from (\ref{schemeC2}) gives
\[ \langle \nabla (c_h^{n+1}-c_h^n), \nabla v_h \rangle +\alpha\langle  c_h^{n+1}-c_h^n,  v_h \rangle= \langle  \rho^{n+1}-\rho^n,  v_h \rangle,\quad \forall  v_h\in V^h.\]
By setting $v_h=-(c_h^{n+1}-c_h^n)\in V^h$, we obtain
\begin{equation*}
\resizebox{\hsize}{!}{$ -\langle  \rho^{n+1}-\rho^n,  c^{n+1}-c^n \rangle=-\langle \nabla (c^{n+1}-c^n), \nabla (c^{n+1}-c^n) \rangle-\alpha\langle  c^{n+1}-c^n,  c^{n+1}-c^n \rangle \leq 0.$}
\end{equation*}
On the other hand, choosing $v_h=c_h^{n+1}$ in (\ref{schemeC1}) and $v_h=c_h^{n}$ in (\ref{schemeC2}) and subtracting both, we obtain
\begin{equation*}
\langle  \rho^n,  c_h^{n+1} \rangle=\langle  \rho^{n+1},  c_h^n \rangle.
\end{equation*}
Therefore,
\begin{equation*}
II=\langle \rho^{n+1},c_h^n\rangle-\frac{1}{2}\langle \rho^{n},c_h^n\rangle-\frac{1}{2}\langle \rho^{n+1},c_h^{n+1}\rangle
=-\frac{1}{2}\langle \rho^{n+1}-\rho^n, c_h^{n+1}-c_h^n\rangle\leq0.
\end{equation*}
\end{proof}


 \section{Numerical tests}
 \label{sec-test}
 
In this section we provide numerical examples to demonstrate the performance of the proposed schemes. We will mainly focus on the Keller-Segel equation as it is more challenging than the Fokker-Planck equation. But one example about the Fokker-Planck equation will be included.

We consider the Keller-Segel system in  a square domain $\Omega$ with a source term:
 \[\begin{cases}
 \partial_t\rho =\Delta \rho-\nabla \cdot( \rho \nabla c)+f(x,y),\\
 -\Delta c+c=\rho,
 \end{cases}\]
 with homogeneous Neumann boundary conditions $\nabla\rho\cdot\mathbf{n}|_{\partial \Omega}=
 \nabla c\cdot\mathbf{n}|_{\partial \Omega}=0$. It is straightforward to verify that the system above is equivalent to 
 \begin{equation}\begin{cases}
 \partial_t\rho =\nabla \cdot (\mathcal M \nabla \frac{\rho}{\mathcal M})+f(x,y),\quad \mathcal M:=e^{c },\\
 -\Delta c+c=\rho,
 \end{cases}
 \label{keller-segel-eqn}
 \end{equation}
 with boundary conditions $\nabla c\cdot\mathbf{n}|_{\partial \Omega}=0$ and $\nabla \frac{\rho}{\mathcal M}\cdot \mathbf n |_{\partial \Omega}=0$.
 We test the second order and fourth order semi-implicit finite difference schemes for solving \eqref{keller-segel-eqn}. 
 
 \subsection{Accuracy test for the Keller-Segel system with a source term}
 
 The proposed semi-implicit schemes can be at most first order accurate in time. For testing the spatial accuracy, we consider
an initial condition $\rho(0,x,y)=3\cos x\cos y+3$, $c(0,x,y)=\cos x\cos y+3$ on $\Omega=(0,\pi)\times(0,\pi)$ and a source term $f(x,y)=-3\cos(2x)\cos^2 y-3\cos^2 x\cos(2y)$ so that the exact solution is a steady state solution.  
The time step is set as $\Delta t=\Delta x$ and errors at $T=1$ are given 
in Table \ref{table-accuracy} where $l^2$ error is defined as 
$$\sqrt{\Delta x \Delta y\sum_i \sum_j |u_{ij}-u(x_i, y_j)|^2}$$ with $u_{ij}$ and $u(x,y)$ denoting the numerical and exact solutions, respectively. 
We observe the expected order of spatial accuracy. 
\begin{table} [htbp]
    \resizebox{\textwidth}{!}{%
\begin{tabular}{|c|c|c|c|c|c|c|c|c|}
\hline
\multirow{2}{*}{FD Grid} & 
\multicolumn{4}{c|}{the second order scheme}  &
\multicolumn{4}{c|}{the fourth order scheme} \\ 
\cline{2-9}
 & 
 $l^2$ error & order & $l^\infty$ error & order  &
$l^2$ error & order & $l^\infty$ error & order  \\
\hline 
$9\times 9$  & 
    2.09E-1     & - &     2.51E-1     & - & 
    1.37E-2     & - &     1.08E-2     & - 
 \\
\hline
$17\times 17$ & 
    4.11E-2     &     2.34     &     6.82E-2     &     1.89     & 
    7.70E-4     &     4.16     &     1.32E-3     &     3.03     
 \\
\hline
$33\times 33$ &
    8.19E-3     &     2.33     &     1.70E-2     &     2.00     & 
    4.52E-5     &     4.09    &     9.72E-5     &     3.76    
 \\
\hline
$65\times 65$ &
    1.77E-3     &     2.21     &     4.29E-3     &     1.99    & 
    2.76E-6     &     4.03     &     6.41E-6     &     3.92     
 \\ \hline
 $129\times 129$ &
    4.04E-4     &     2.13     &     1.08E-3     &     1.99     & 
    1.71E-7     &     4.01     &     4.09E-7     &     3.97     
 \\ \hline
\end{tabular}}
\label{table-accuracy}
\caption{Accuracy test for the Keller-Segel system with a source term.}
\end{table}

\subsection{A steady state solution of the Fokker-Planck equation}

We now  test the second order and fourth order schemes for solving the following two-dimensional linear Fokker-Planck equation on $\Omega=(-3,3)\times(-3,3)$:
 \begin{equation} 
 \partial_t\rho=\Delta \rho+\nabla \cdot( \rho \nabla \mV), \quad \mV=\frac{x^2+y^2}{2}.
 \end{equation}
 It is equivalent to 
 \[\partial_t\rho =\nabla \cdot (\mathcal M \nabla \frac{\rho}{\mathcal M}), \quad \mathcal M:=e^{-\frac{x^2+y^2}{2}},
 \]
 with the boundary condition $\nabla \frac{\rho}{\mathcal M}\cdot \mathbf n |_{\partial \Omega}=0$.
This equation admits an exact solution:
\[\rho(t,x,y)=\frac{1}{2\pi (1-e^{-2t})}e^{-\frac{x^2+y^2}{2(1-e^{-2t})}}.\]
We use   $\rho(1,x,y)$  as an initial condition and march to time $T=20$ for approximating 
the steady state 
\[\rho_\infty(x,y)=\frac{1}{2\pi }e^{-\frac{x^2+y^2}{2}}.\]

To demonstrate the advantages of our schemes, we also compare them to the second order spatial discretization  with fully explicit forward Euler time discretization, which can also
be proven positivity-preserving and energy-dissipative but under a small time step constraint $\Delta t=\mathcal O(\Delta x^2)$. 
In Figure \ref{FP}, we can see that the convergence of the explicit scheme to the steady state solution is much slower. Moreover, the small time step $\Delta t=\mathcal O(\Delta x^2)$ is usually
not desired in applications. The convergence to numerical steady state solution of two implicit schemes are similar. On the other hand, the fourth order scheme produces slightly smaller errors in the numerical steady state solution. 

In Figure \ref{FP}, after $T=10$, steady state solution errors of both implicit schemes stay flat, and in each time step $\|\rho^{n+1}-\rho^n\|_\infty$ is less than $10^{-10}$, which is the accuracy tolerance of preconditioned conjugate gradient linear system solver. 
At $T=20$, compared to the exact steady state, the fourth order scheme with implicit time stepping produces error in discrete 2-norm as $8.18\times 10^{-4}$ and 
the second order scheme with implicit time stepping produces error in discrete 2-norm $8.35\times 10^{-4}$. We emphasize both implicit schemes are used on the same grid and the difference in computational cost is marginal, thus this is a clear advantage
of using a high order accurate spatial discretization, even if the time accuracy is only first order.

      \begin{figure}[htbp]
      \begin{center}
 \subfigure[Three schemes are used on the same $33\times 33$ grid. The implicit schemes use  a time step $\Delta t=\mathcal O(\Delta x)$
and the explicit scheme uses a time step $\Delta t=\mathcal O(\Delta x^2)$. ]{\includegraphics[scale=0.35]{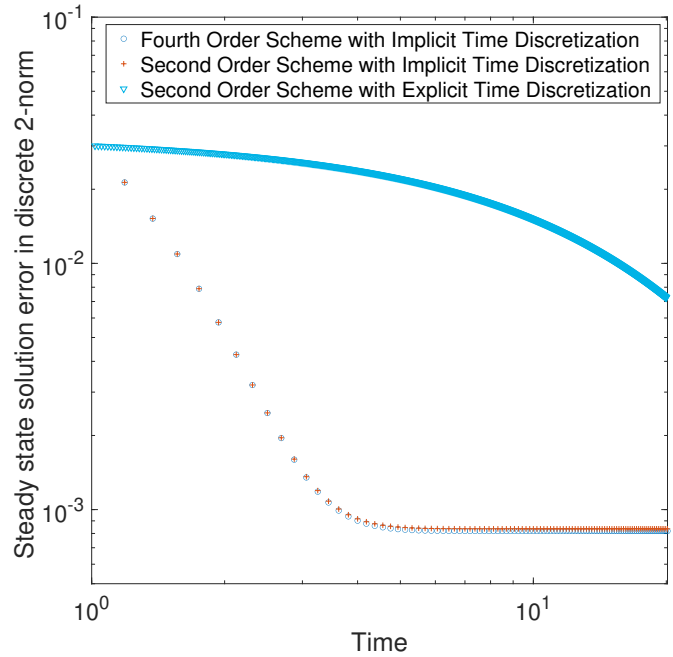} }
 \hspace{.1in}
 \subfigure[The steady state solution. Numerical solution was generated by the fourth order scheme on a $33\times 33$ grid. ]{\includegraphics[scale=0.4]{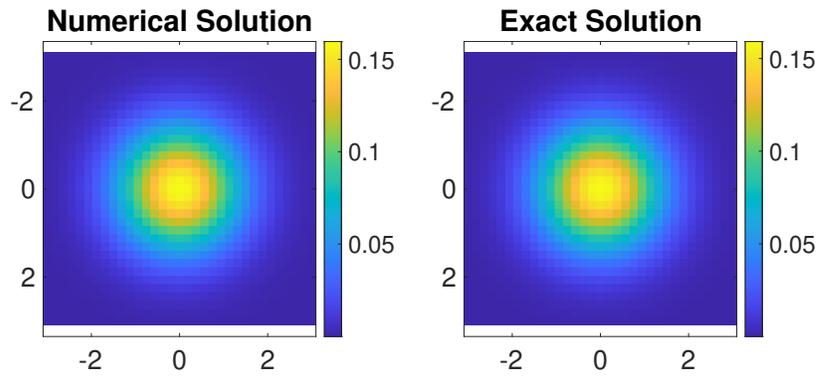}}
 \caption{Linear Fokker-Planck equation on $\Omega=(-3,3)\times(-3,3)$. }
\label{FP}
\end{center}
 \end{figure}

 \subsection{A smooth solution of the Keller-Segel system}

For the Keller-Segel system,  it is well-known that there is a critical value for total mass in initial conditions, below which a globally  well-posed solution exists \cite{dolbeault2004optimal, blanchet2006two}.  
 We solve the system \eqref{keller-segel-eqn} with $f(x,y)\equiv 0$ on $ \Omega=(-2, 2)\times(-2, 2)$ with an initial condition $\rho(0,x,y)=\frac{60}{1+40(x^2+y^2)}$ and its mass is below the critical value. See both schemes on the same grid of $101\times 101$ points  at $T=2$ in Figure \ref{KS-smooth}. 
   For both schemes, $\Delta t=\Delta x$ is used. 
Then we run two schemes for longer time until  $\|\rho^{n+1}-\rho^n\|_\infty\leq 10^{-8}$ is satisfied. Both schemes reach $\|\rho^{n+1}-\rho^n\|_\infty\leq 10^{-8}$ 
around $T=13.52.$ See numerical solutions  at $T=13.52$ in Figure \ref{KS-smooth2}. Note that in this case, the energy as defined in \eqref{energy-definition} reaches a constant value which is an indicator that the system has already reached the steady state.

  \subsection{A blow-up solution of the Keller-Segel system}
 
 For an initial condition with total mass above the critical mass, a blow-up will emerge in finite time for the Keller-Segel system \cite{dolbeault2004optimal, blanchet2006two}, see also  \cite{CCY19, GLY19} for computational examples. 
 
We test both schemes for an initial condition $\rho(0,x,y)=\frac{100}{1+40(x^2+y^2)}$ with total mass above the critical value.  
See solutions at $T=0.11$ in Figure \ref{KS-blowup-earliest},
 at $T=0.2$ in Figure \ref{KS-blowup-early}
 and at $T=0.8$ in Figure \ref{KS-blowup}.
  For both schemes, $\Delta t=\Delta x$ is used. Note that at $T=0.8$, the solution in the fourth order scheme is significantly different from the second order one, while the former is certainly more faithful due to its higher accuracy.

The energy evolution of numerical solutions is shown in Figure \ref{KS-energy}, where the discrete energy is defined as in \eqref{energy-definition}. It should be mentioned that the mesh constraints in Section \ref{sec-mono} for achieving monotonicity in the fourth order scheme will be eventually impossible to be satisfied for a blow-up solution, yet these mesh constraints are only sufficient conditions for monotonicity. In our fourth order numerical solutions, it has been checked that $\rho$ is always positive even after blow up. Therefore, the energy dissipation is still in good faith.

    \begin{figure}[htbp]
 \subfigure[The second order scheme.]{\includegraphics[scale=0.3]{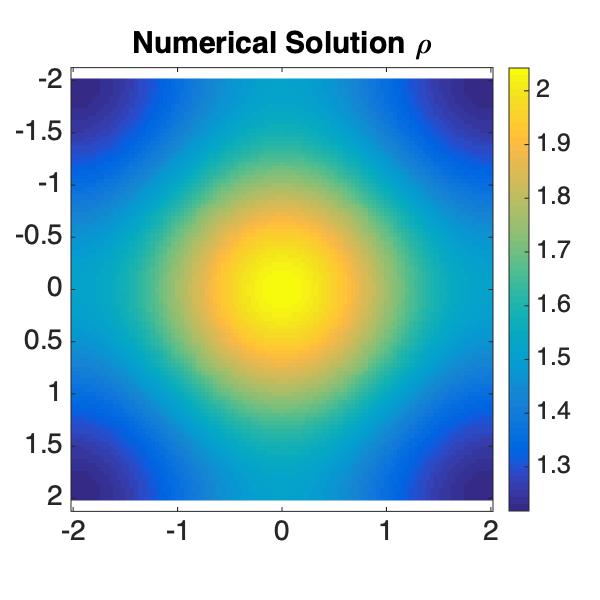} }
 \hspace{.1in}
 \subfigure[The fourth order scheme.]{\includegraphics[scale=0.3]{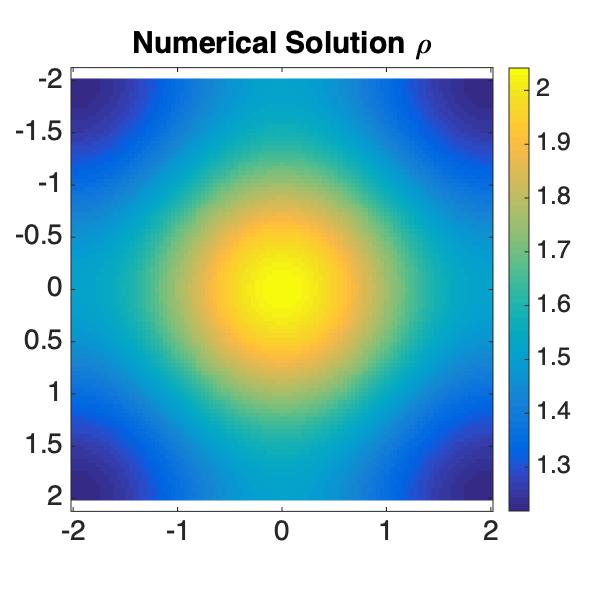}}\\
  \subfigure[The second order scheme.]{\includegraphics[scale=0.3]{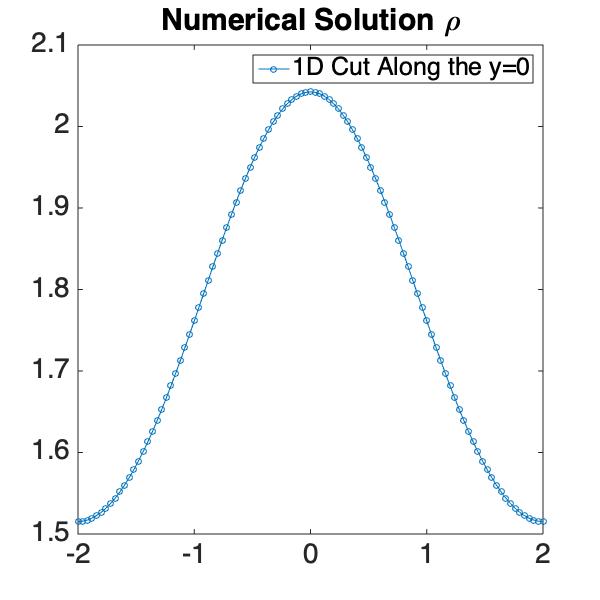} }
 \hspace{.1in}
 \subfigure[The fourth order scheme.]{\includegraphics[scale=0.3]{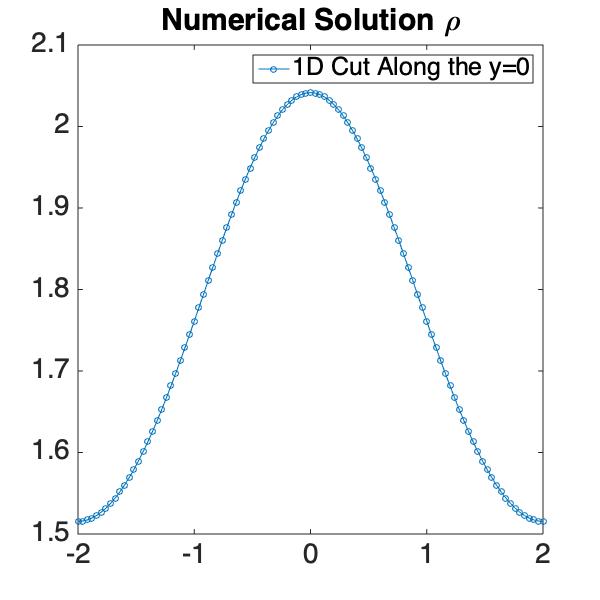}}\\
\caption{Keller-Segel system with an initial condition below critical mass $\rho(x,y,0)=\frac{60}{1+40(x^2+y^2)}$ on $\Omega=(-2,2)\times(-2,2)$. 
The solutions at $T=2$ are plotted. Both schemes are computed on a $101\times 101$ grid.}
\label{KS-smooth}
 \end{figure}

      \begin{figure}[htbp]
 \hspace{-.3in} \subfigure[The second order scheme.]{\includegraphics[scale=0.25]{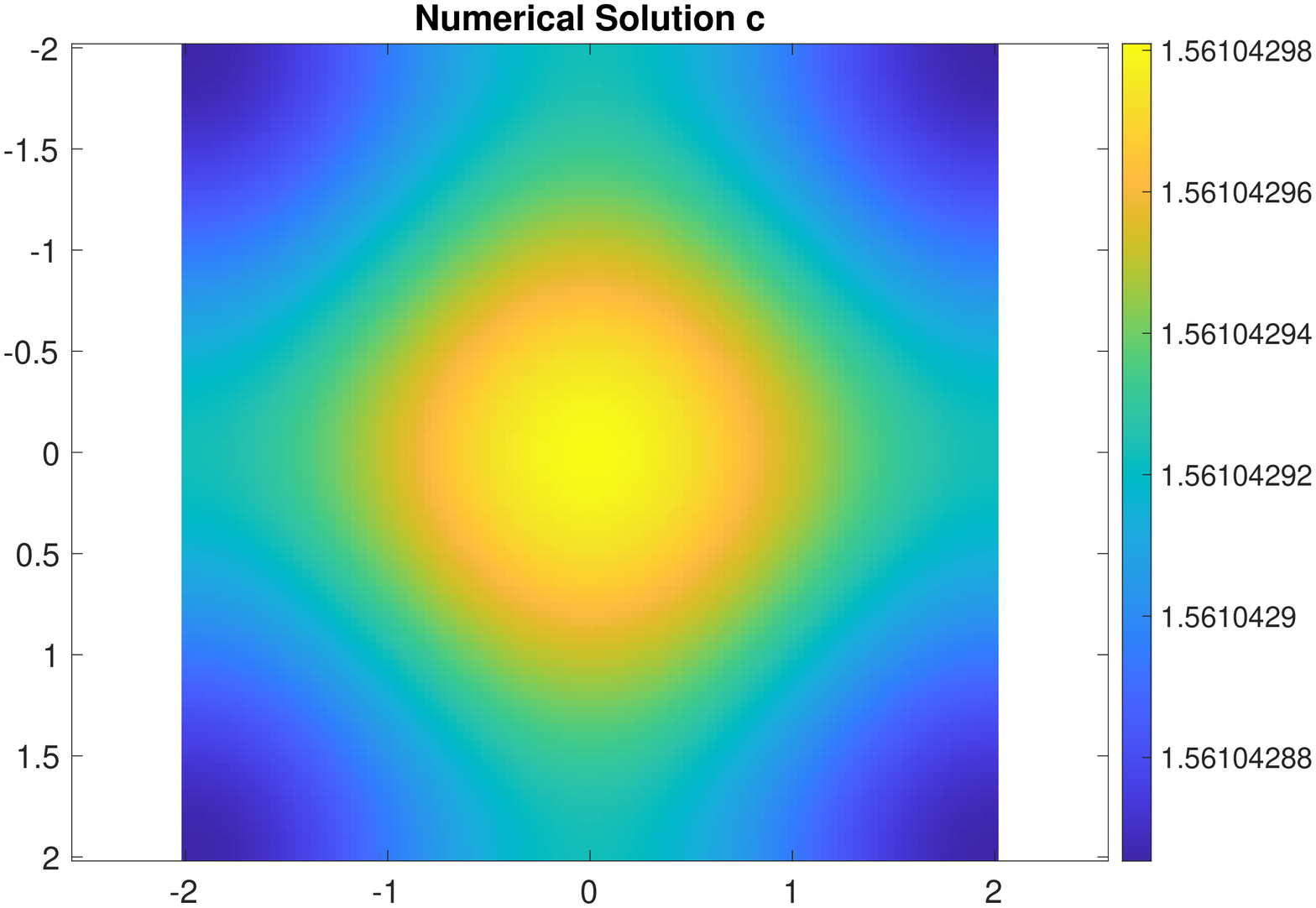} }
 \subfigure[The fourth order scheme.]{\includegraphics[scale=0.25]{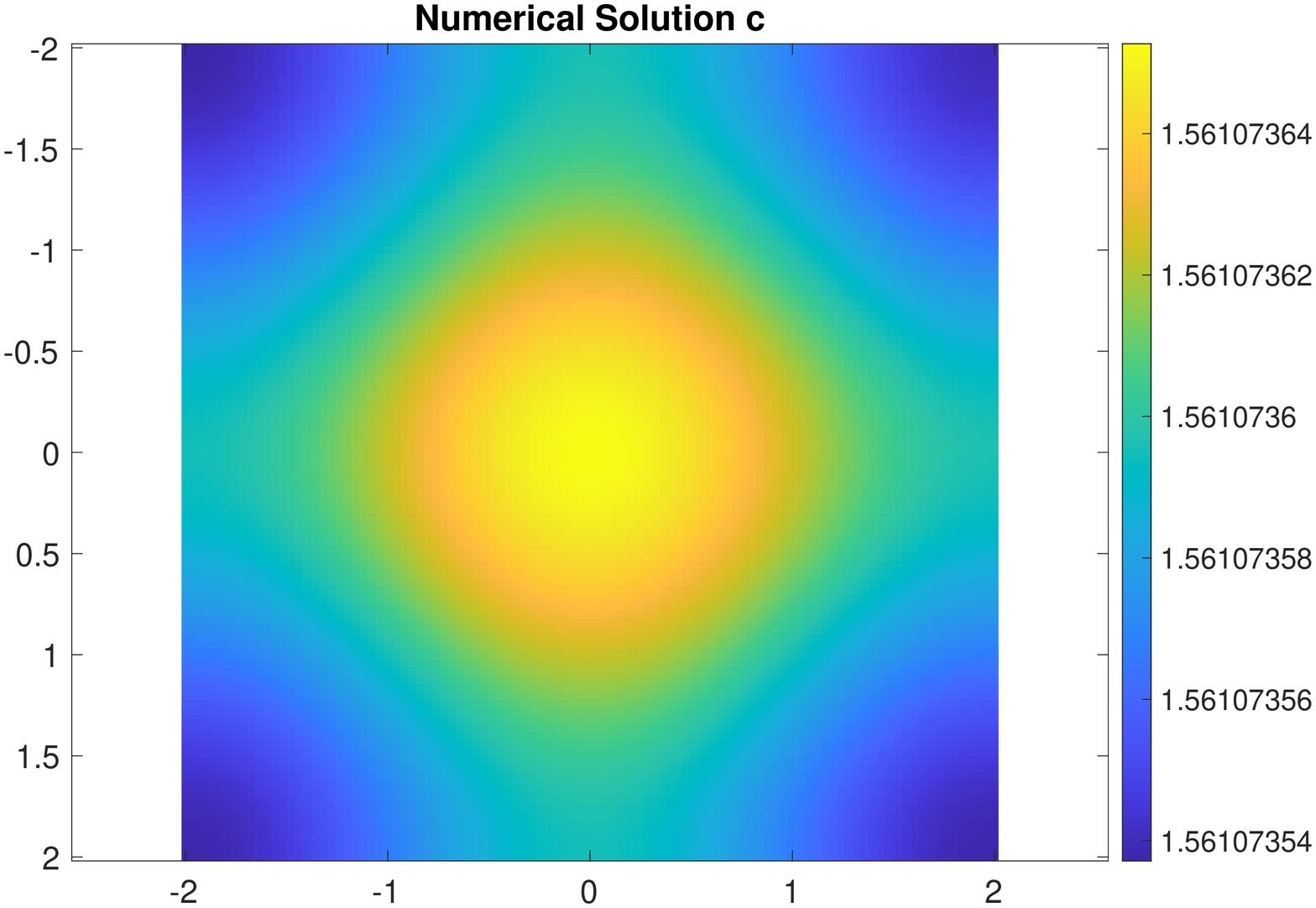}}\\ 
 
 \hspace{-.3in}\subfigure[The second order scheme.]{\includegraphics[scale=0.25]{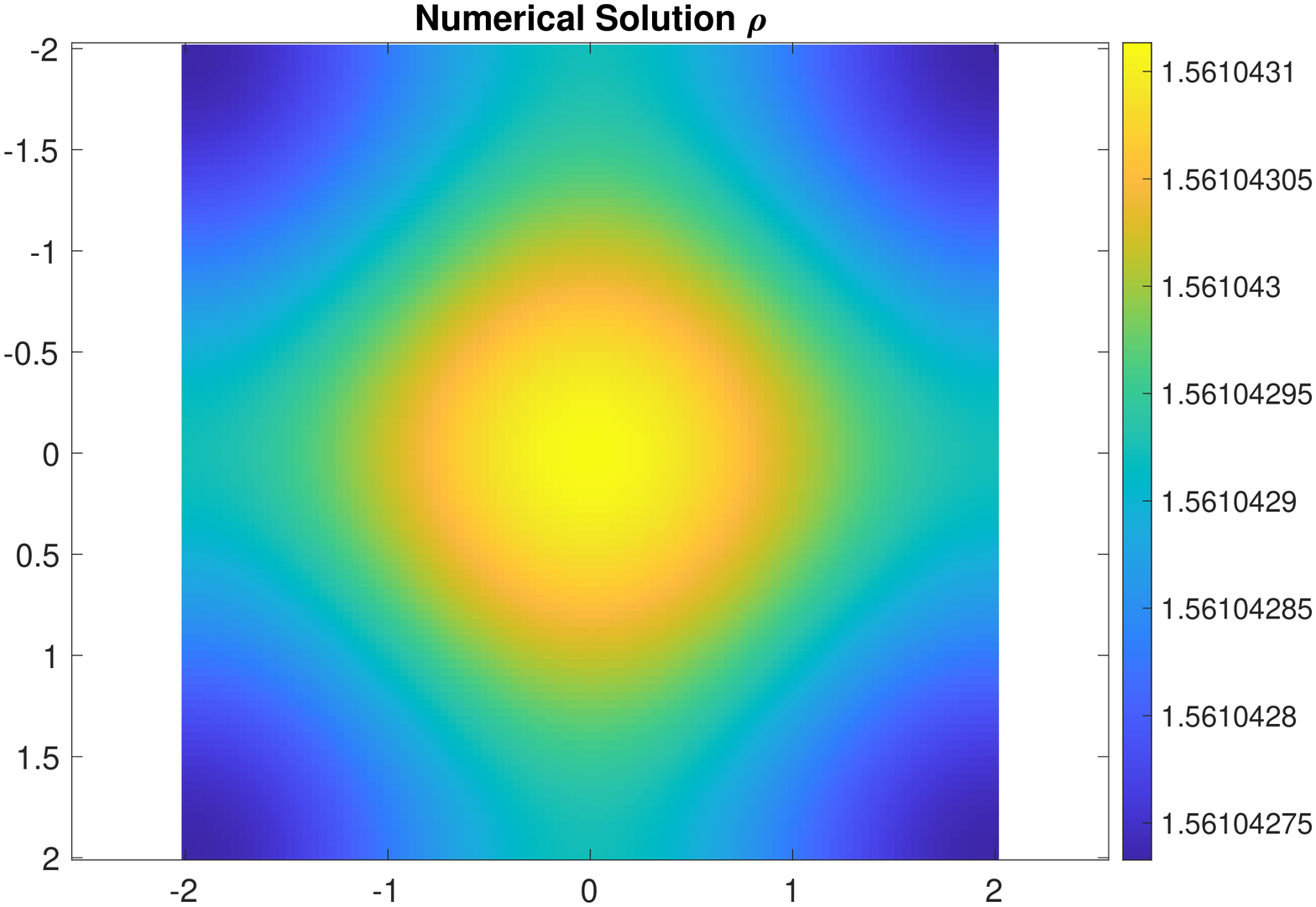} }
\subfigure[The fourth order scheme.]{\includegraphics[scale=0.25]{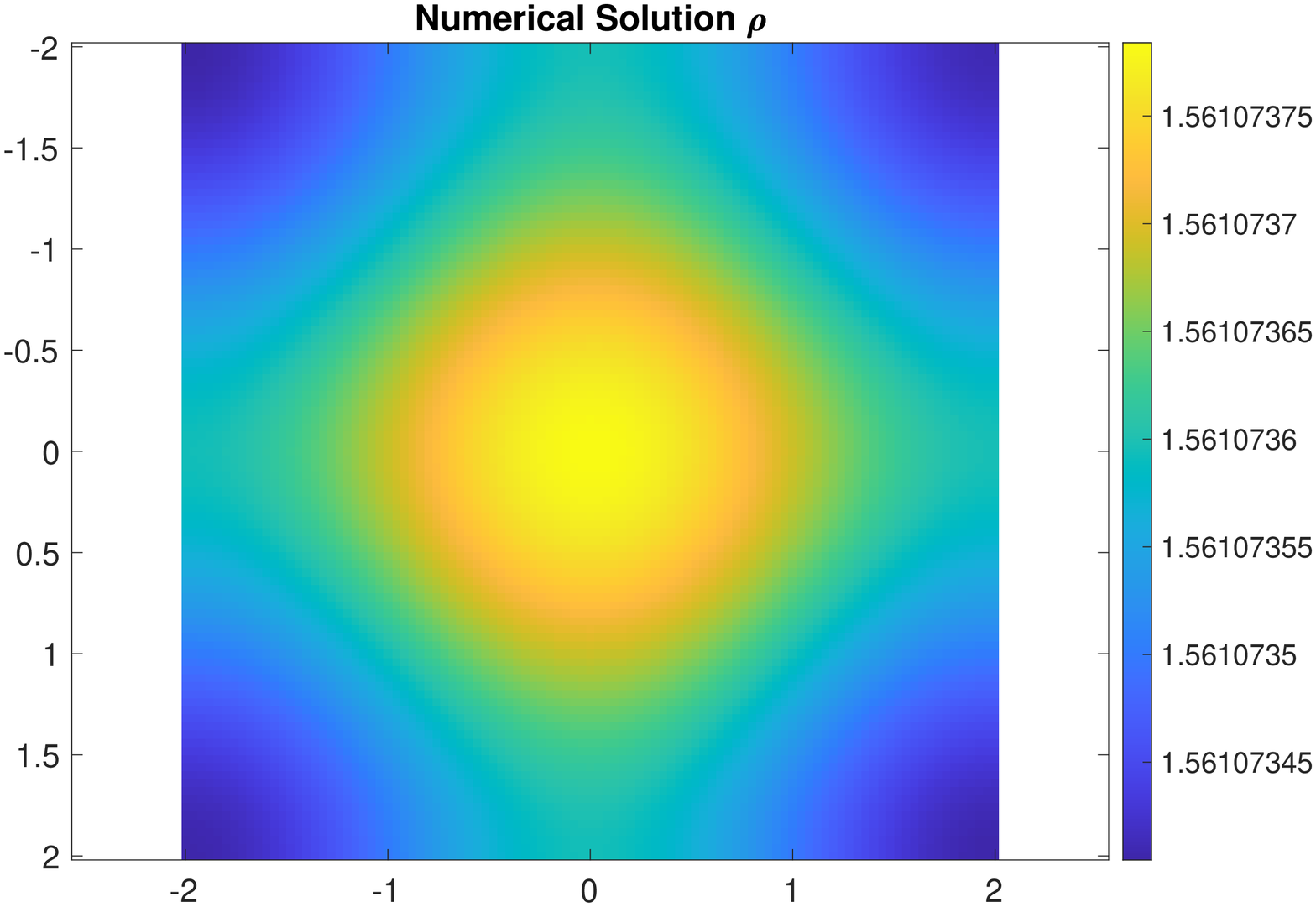}}\\
 
  \subfigure[The second order scheme.]{\includegraphics[scale=0.25]{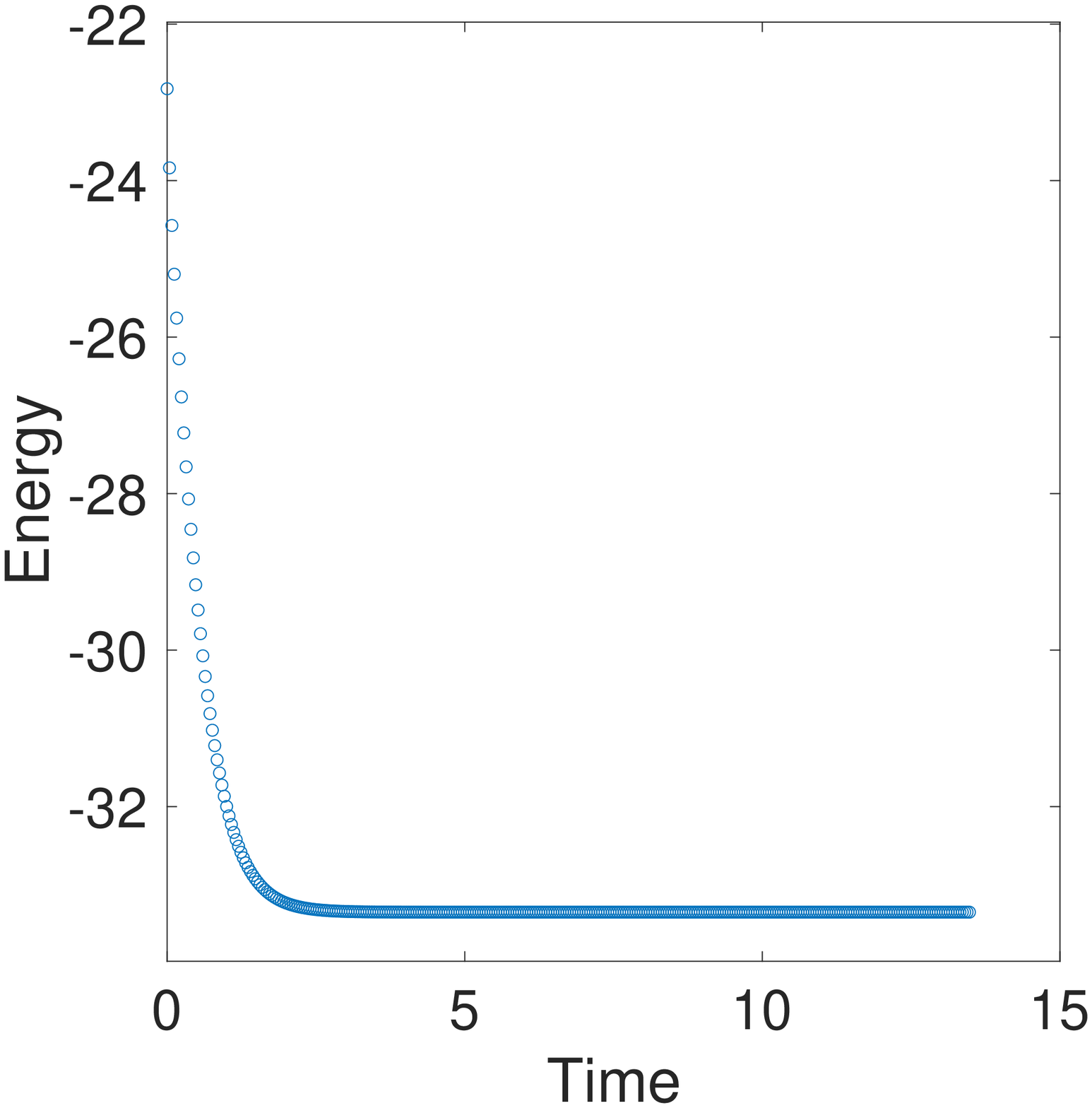} }
 \hspace{.5in}
 \subfigure[The fourth order scheme.]{\includegraphics[scale=0.25]{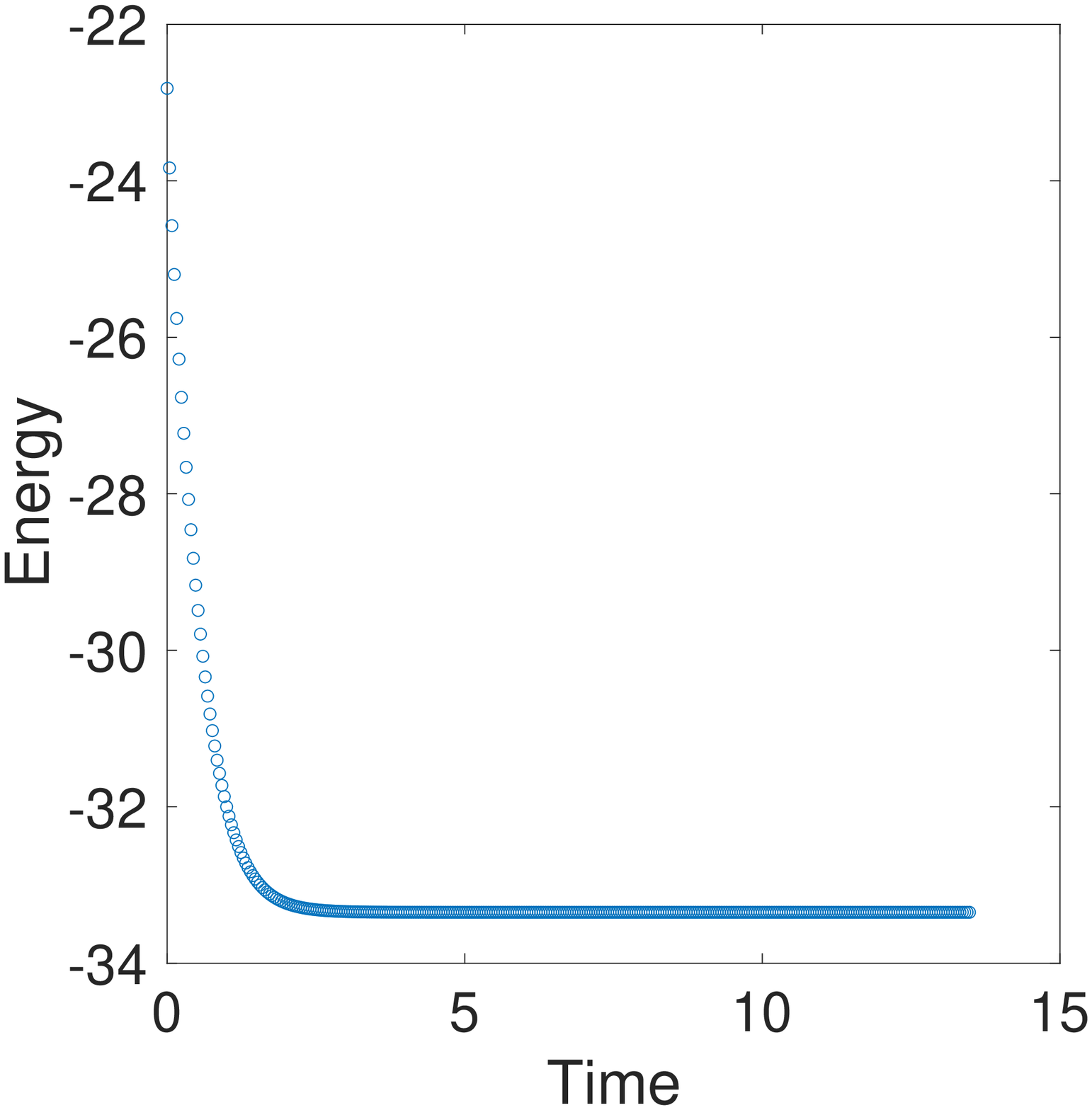}}\\  
\caption{Keller-Segel system with an initial condition below critical mass $\rho(x,y,0)=\frac{60}{1+40(x^2+y^2)}$. 
The plotted numerical solutions are around the time $T=13.52$ when $\|\rho^{n+1}-\rho^n\|_\infty\leq 10^{-8}$. Both schemes are computed on a $101\times 101$ grid.}
\label{KS-smooth2}
 \end{figure}

   \begin{figure}[htbp]
 \subfigure[The second order scheme  at $T=0.11$.]{\includegraphics[scale=0.3]{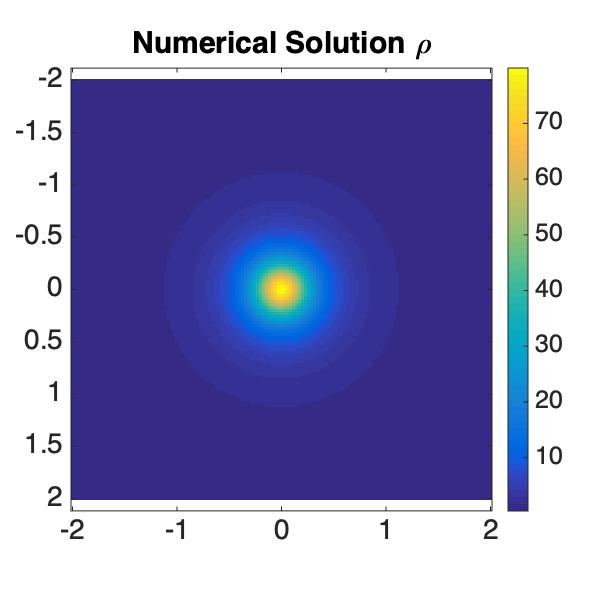} }
 \hspace{.1in}
 \subfigure[The fourth order scheme at $T=0.11$.]{\includegraphics[scale=0.3]{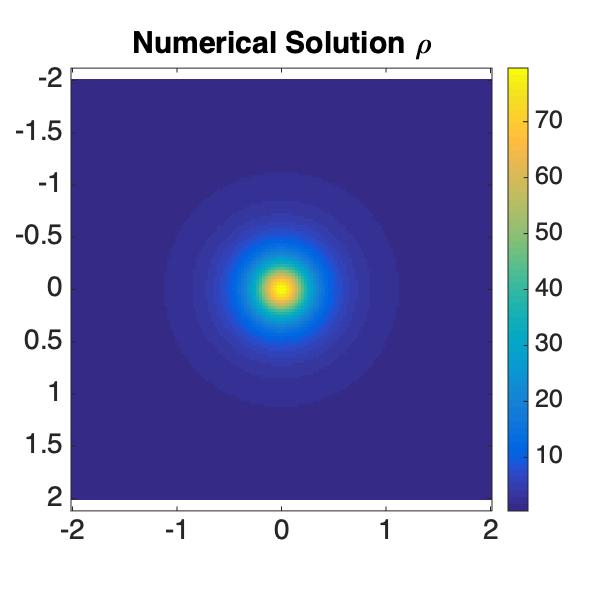}}\\
  \subfigure[The second order scheme at $T=0.11$.]{\includegraphics[scale=0.3]{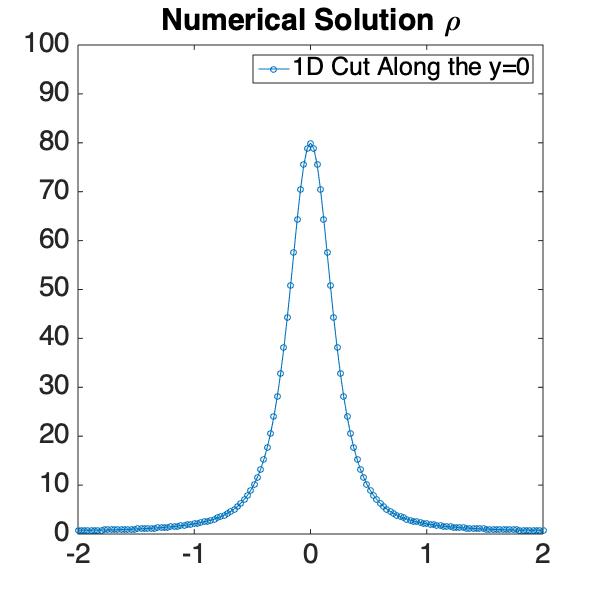} }
 \hspace{.1in}
 \subfigure[The fourth order scheme at $T=0.11$.]{\includegraphics[scale=0.3]{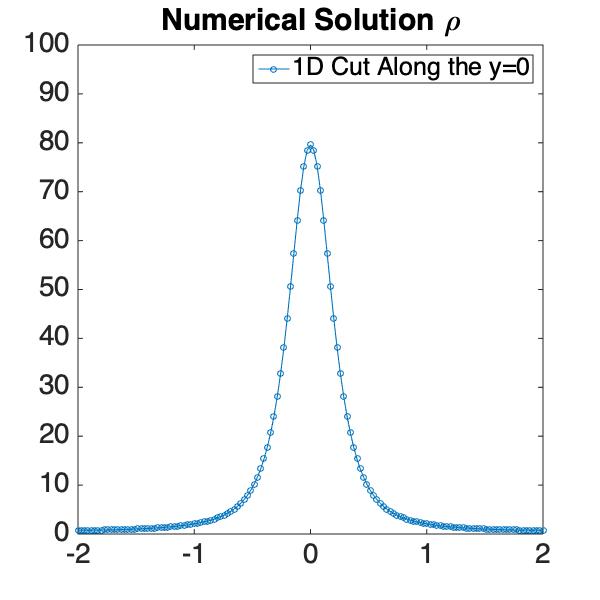}}
 \caption{Keller-Segel system with an initial condition above critical mass $\rho(x,y,0)=\frac{100}{1+40(x^2+y^2)}$ on $\Omega=(-2,2)\times(-2,2)$. 
Both schemes are computed on a $141\times 141$ grid.}
\label{KS-blowup-earliest}
 \end{figure}

    \begin{figure}[htbp]
 \subfigure[The second order scheme at $T=0.2$.]{\includegraphics[scale=0.3]{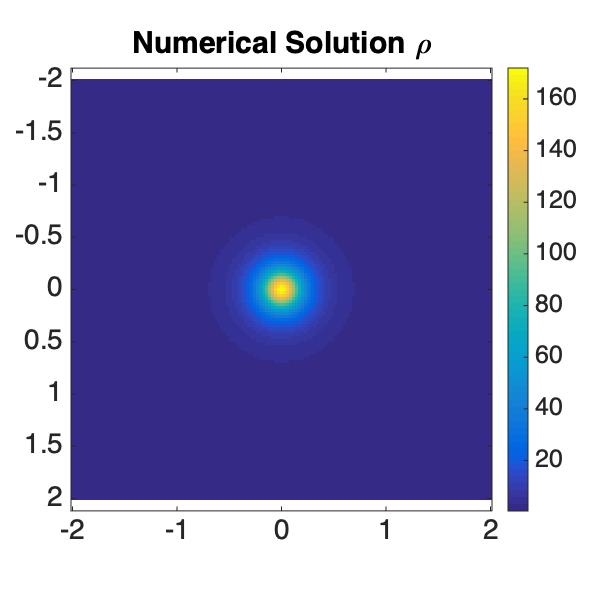} }
 \hspace{.1in}
 \subfigure[The fourth order scheme at $T=0.2$.]{\includegraphics[scale=0.3]{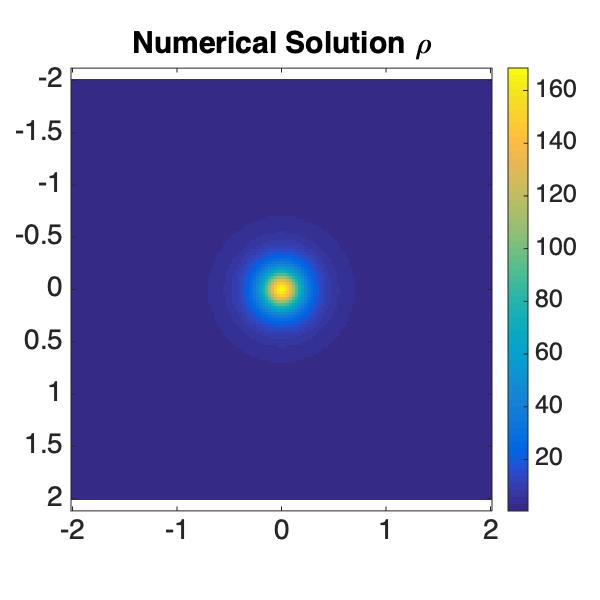}}\\
  \subfigure[The second order scheme at $T=0.2$.]{\includegraphics[scale=0.3]{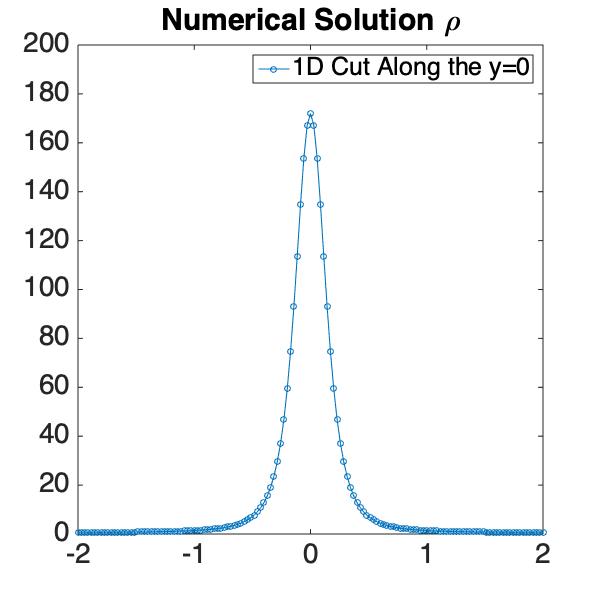} }
 \hspace{.1in}
 \subfigure[The fourth order scheme at $T=0.2$.]{\includegraphics[scale=0.3]{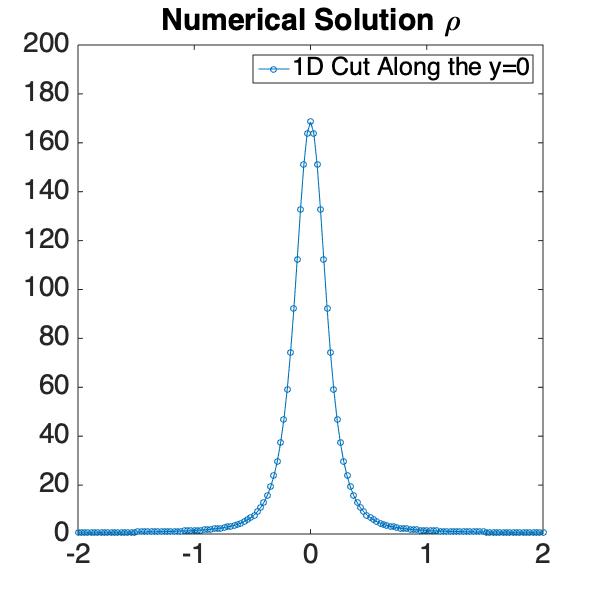}}
 \caption{Keller-Segel system with an initial condition above critical mass $\rho(x,y,0)=\frac{100}{1+40(x^2+y^2)}$ on $\Omega=(-2,2)\times(-2,2)$. 
 Both schemes are computed on a $141\times 141$ grid.}
\label{KS-blowup-early}
 \end{figure}

     \begin{figure}[htbp]
 \subfigure[The second order scheme at $T=0.8$.]{\includegraphics[scale=0.3]{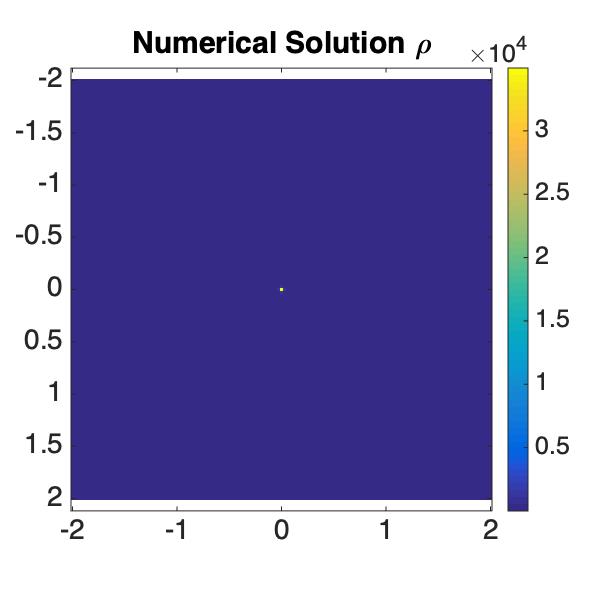} }
 \hspace{.1in}
 \subfigure[The fourth order scheme at $T=0.8$.]{\includegraphics[scale=0.3]{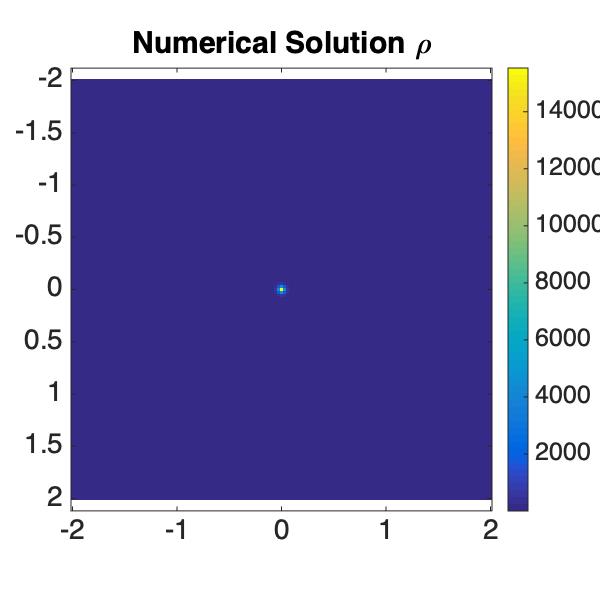}}\\
  \subfigure[The second order scheme at $T=0.8$.]{\includegraphics[scale=0.3]{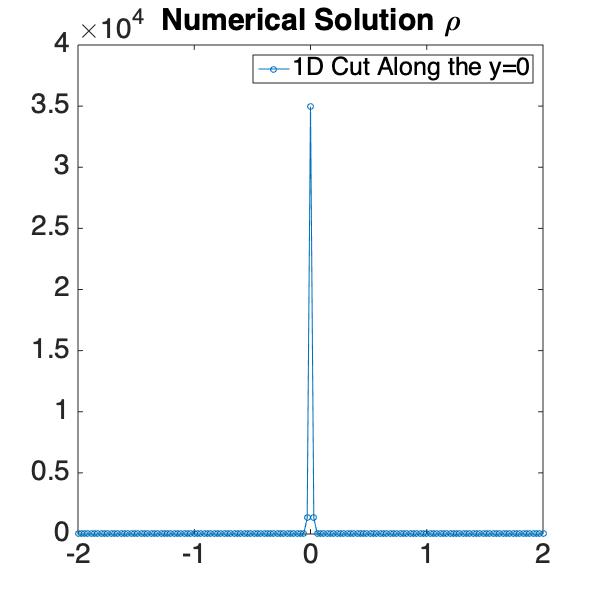} }
 \hspace{.1in}
 \subfigure[The fourth order scheme at $T=0.8$.]{\includegraphics[scale=0.3]{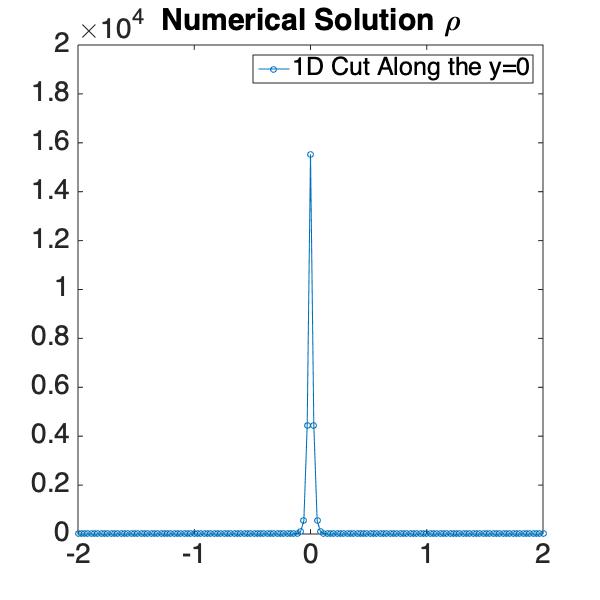}}
 \caption{Keller-Segel system with an initial condition above critical mass $\rho(x,y,0)=\frac{100}{1+40(x^2+y^2)}$ on $\Omega=(-2,2)\times(-2,2)$. 
 Both schemes are computed on a $141\times 141$ grid.}
\label{KS-blowup}
 \end{figure}

      \begin{figure}[htbp]
 \subfigure[The energy evolution  of the second order scheme.]{\includegraphics[scale=0.3]{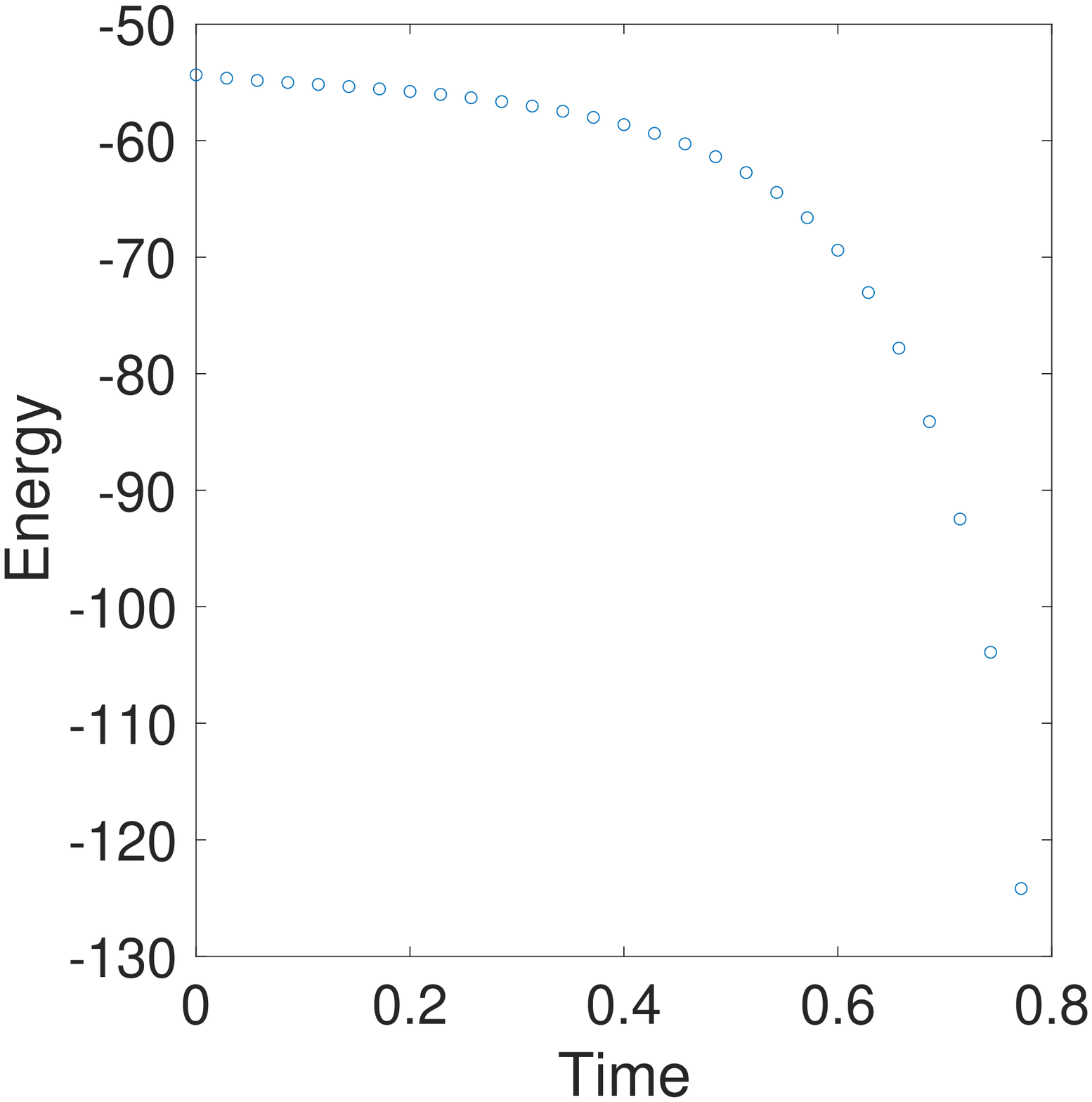} }
 \hspace{.1in}
 \subfigure[The energy evolution  of the fourth order scheme.]{\includegraphics[scale=0.3]{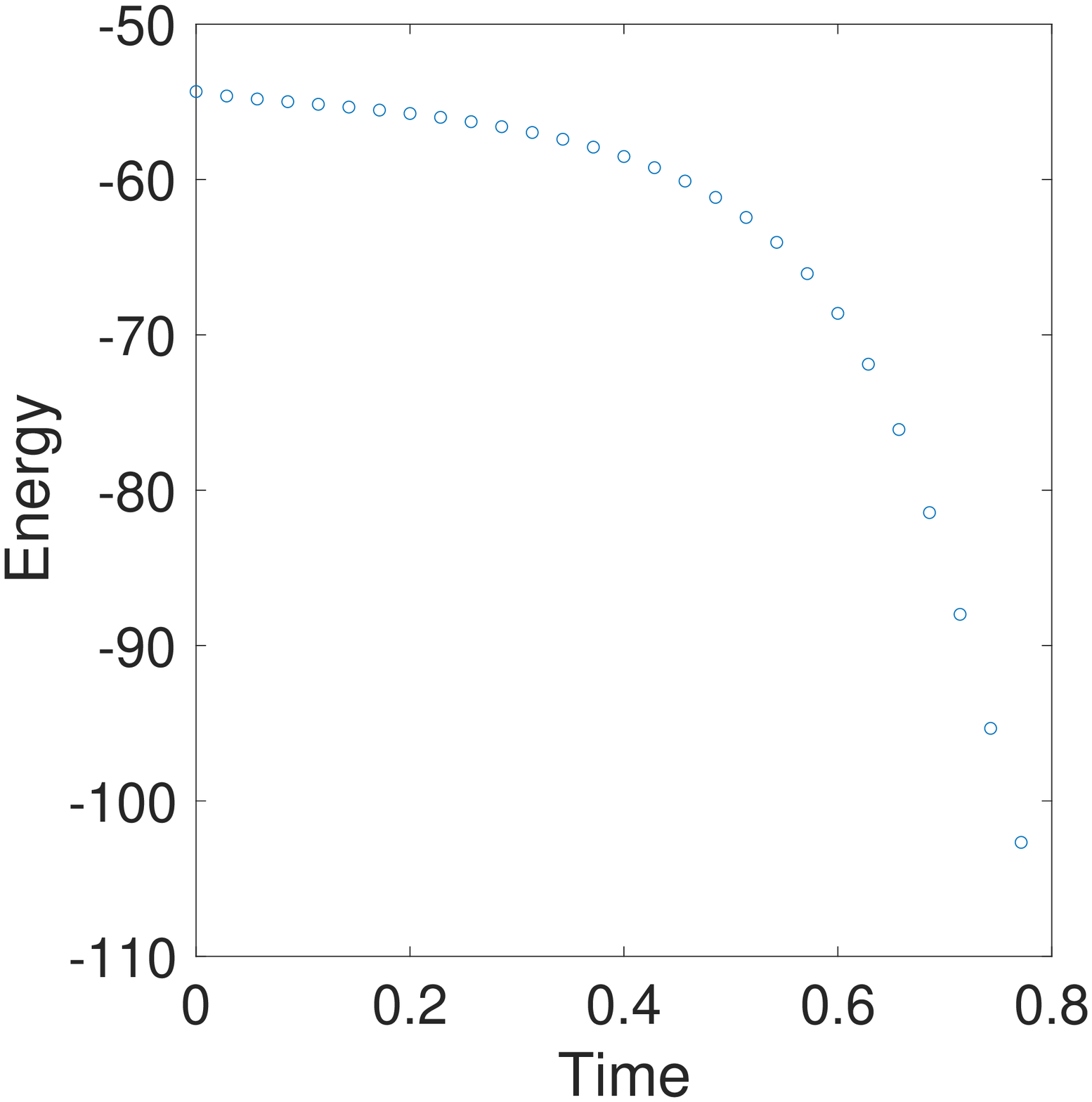}}
 \caption{Keller-Segel system with an initial condition above critical mass $\rho(x,y,0)=\frac{100}{1+40(x^2+y^2)}$ on $\Omega=(-2,2)\times(-2,2)$. 
 Both schemes are computed on a $141\times 141$ grid.}
\label{KS-energy}
 \end{figure}
 
 \section{Concluding remarks}
 \label{sec-remark}
 We have constructed two finite difference schemes which are proved be positivity-preserving and energy-dissipative for the Fokker-Planck and Keller-Segel type equations. The time discretization is a first order semi-implicit or implicit scheme. The spatial discretizations include a second order and a fourth order finite difference scheme, obtained via finite difference implementation of the finite element method with linear and quadratic polynomials on uniform meshes. Under mild mesh size and time step constraints for smooth solutions (a lower bound on time step rather than upper bound), the fourth order scheme is proved to be monotone thus is positivity-preserving and decays energy, which is  the first high order spatial discretization with these properties. 
Numerical tests on both the Fokker-Planck equation and Keller-Segel system are performed to verify the performance of the proposed schemes.  
 
\bibliographystyle{plain}
\bibliography{ref.bib}

 \end{document}